\documentclass[a4paper,10pt]{amsart}
\usepackage[utf8]{inputenc}
\usepackage{amsthm}
\usepackage{amsmath}
\usepackage{mathabx}
\usepackage{bm}
\usepackage{amsfonts}
\usepackage{amssymb}
\usepackage{mathtools}
\usepackage{appendix}
\usepackage{mathrsfs}
\usepackage{setspace}
\usepackage{xcolor}
\usepackage{todonotes}
\usepackage{thmtools} % solves problem of shared counters in statements and autoref
\usepackage{csquotes}
\usepackage{bbm}
\usepackage[foot]{amsaddr}
\RequirePackage[a4paper,top=2.54cm,bottom=2.54cm,left=1.90cm,right=1.90cm,%
                headsep=1em,includehead,includefoot]{geometry}
\usepackage[pdfdisplaydoctitle,colorlinks,breaklinks,urlcolor=blue,linkcolor=blue,citecolor=blue]{hyperref} %must always be the last

\newcommand{\E}{\mathcal{E}}
\newcommand{\EE}{\mathbb{E}}

\newcommand{\G}{\mathcal{G}}

\newcommand{\HH}{\mathcal{H}}

\newcommand{\R}{\mathbb{R}}

\let\div\relax

\DeclareMathOperator{\div}{div}

\renewcommand{\epsilon}{\varepsilon}
\renewcommand{\setminus}{\smallsetminus}

\newtheorem{theorem}{Theorem}[section]
\newtheorem{definition}[theorem]{Definition}
\newtheorem{hypothesis}[theorem]{Hypothesis}

\newtheorem{lemma}[theorem]{Lemma}
\newtheorem{proposition}[theorem]{Proposition}
\newtheorem{conjecture}[theorem]{Conjecture}

\theoremstyle{remark}
\newtheorem{remark}[theorem]{Remark}

\numberwithin{equation}{section}

\renewcommand{\P}{\mathbb{P}}
\newcommand{\0}{{\mathcal{H}_0}}

\newcommand{\eps}{\varepsilon}

\DeclareMathOperator{\grad}{\nabla}
\newenvironment{acknowledgements}{%
  \begin{abstract}
}{%
  \end{abstract}
}

\title[LDP for Fluid Dynamic Systems in Bounded Domains]{Large Deviations Principle for the Inviscid Limit of Fluid Dynamic Systems in 2D Bounded Domains}
\author[F. Butori]{Federico Butori}
\address{Scuola Normale Superiore, Piazza dei Cavalieri, 7, 56126 Pisa, Italia}
\email{\href{mailto:federico.butori at sns.it}{federico.butori at sns.it}}
\author[E. Luongo]{Eliseo Luongo}
\address{Scuola Normale Superiore, Piazza dei Cavalieri, 7, 56126 Pisa, Italia}
\email{\href{mailto:eliseo.luongo at sns.it}{eliseo.luongo at sns.it}}

\date\today

\begin{document}

\begin{abstract}
	Using a weak convergence approach, we establish a Large Deviation Principle (LDP) for the solutions of fluid dynamic systems in two-dimensional bounded domains subjected to no-slip boundary conditions and perturbed by additive noise. Our analysis considers the convergence of both viscosity and noise intensity to zero. Specifically, we focus on three important scenarios: Navier-Stokes equations in a Kato-type regime, Navier-Stokes equations for fluids with circularly symmetric flows and Second-Grade Fluid equations. In all three cases, we demonstrate the validity of the LDP, taking into account the critical topology $C([0,T];L^2).$
\end{abstract}

\maketitle

%%%%%%%%%%%%%%%%%%%%%%%%%%%%%%%%%%%%%%%%%%%%%%%%%%%%%%%%%%%%%%%%%%%%%%%%%%%%%%%%%%%%%%%%%%%%%
\section{Introduction}\label{sec:introduction}
\subsection{The Problem of the Inviscid Limit}\label{subsec: inviscid limit intro}
A important role in the understanding of the behavior of turbulent fluid is given by the analysis of the so-called inviscid limit. In a naive way, given $u^{NS,\eps}$ and $\Bar{u}$ solutions, in a suitable sense, of the systems below
\begin{align}\label{systems inviscid limit}
        \begin{cases}
        \partial_t u^{NS,\eps}-\eps \Delta u^{NS,\eps}+\nabla p^{NS,\eps}+u^{NS,\eps}\cdot\nabla u^{NS,\eps} & =f^{\eps}\\
        \div u^{NS,\eps}&=0\\
        u^{NS,\eps}(0)&=u_0^{\eps}
    \end{cases}\quad
    \begin{cases}
                \partial_t \bar{u}+\nabla \bar{p}+\bar{u}\cdot\nabla \bar{u}& =\bar{f}\\
        \div \bar{u}&=0\\
        \bar{u}(0)&=\bar{u}_0,
    \end{cases}
\end{align}
the problem of the inviscid limit consists in showing that $u^{NS,\eps}$, the solution of the Navier-Stokes equations, converges to $\Bar{u}$, the solution of the corresponding Euler equations, as $\eps\rightarrow 0$ in the topology $L^{\infty}(0,T;L^2(D))$, $D$ being the domain where the equations evolve. \\
The difficulty of answering to this problem changes drastically considering different boundary conditions.  Some frameworks where the problem above has a positive answer have been presented in \cite{da2009sharp}, \cite{lions1996mathematical}. We sum up their results.
\begin{enumerate}
    \item If previous equations evolve in a two-dimensional domain $D$ without boundary, assuming \begin{align*}
    f^{\eps}= \bar{f}\equiv 0,\ \ u_0^{\eps},\ \bar{u}_0\in H^{k},\ \  \div (u_0^{\eps})=\div (\bar{u}_0)=0,\ \  \lVert u_0^{\eps}-\bar{u}_0\rVert_{H^k}\stackrel{\eps\rightarrow 0}{\rightarrow} 0
\end{align*}  
implies the convergence 
\begin{align*}
    \lVert u^{NS,\eps}- \bar{u}\rVert_{C([0,T];H^k)}\stackrel{\eps\rightarrow 0}{\rightarrow} 0.
\end{align*}
\item If $u^{NS,\eps}$ is a solution of the Navier-Stokes equations with Navier-Boundary conditions, i.e. $
    u^{NS,\eps}\cdot n|_{\partial D}=0,\ \partial_1 u^{NS,\eps}_2-\partial_2 u^{NS,\eps}_1|_{\partial D}=0$, 
    \begin{align*}
        u_0^{\eps}\stackrel{L^2(D)}{\rightarrow}\Bar{u}_0, \quad f^{\eps}\stackrel{L^2((0,T)\times D)}{\rightarrow}\Bar{f},
    \end{align*} then each sequence, $u^{\eps_k}$, has a subsequence, $u^{\eps_{h_k}}$, converging to a weak solution of the Euler equations in $C([0,T];L^2(D))$. Moreover, if the solution of the Euler equation is unique, then $u^{NS,\eps}\rightarrow \Bar{u}$.
\end{enumerate}
In the case of no-slip boundary conditions, i.e. $u^{NS,\eps}|_{\partial D}=0$, the convergence of $u^{NS,\eps}$ to $\Bar{u}$ in the topology $L^{\infty}(0,T;L^2(D))$ is an open problem with few results available: \begin{enumerate}
    \item Unconditioned results. They are based on strong assumptions about the symmetry of the domain and of the data \cite{lopes2008vanishing}, or real analytic data \cite{sammartino1998zero}, \cite{bardos2021inviscid}.
    \item Conditioned results. They are based on stating some criteria about the behavior of the solutions of the Navier-Stokes equations in the boundary layer in order to prove that the inviscid limit holds. This line of research started with the seminal work by Kato \cite{kato1984remarks}, see also \cite{constantin2015inviscid}, \cite{temam1997behavior}, \cite{wang2001kato}, \cite{kelliher2023strong} and the references therein for other results in this direction.
\end{enumerate}
Beside its mathematical interest, the analysis of the inviscid limit in the case of no-slip boundary conditions is a relevant problem also from a physical prospective of the understanding of turbulence. 
The no-slip condition $u^{NS,\eps}|_{\partial D}=0$ provokes large stress
near the boundary, if $u^{NS,\eps}$ is large nearby and this stress, when the viscosity
is small enough, may lead to instabilities and generate vortices. This is the so-called phenomenon of the emergence of a boundary layer: close to the boundary the fluid presents a turbulent behavior for $\eps\rightarrow 0$. The thickness of the boundary layer and some control on the behavior of the fluid in this region are very challenging and mostly open questions, see \cite{bardos2013mathematics} for a review on the topic. \\
Contrary to Navier-Stokes equations, the problem of the inviscid limit in bounded domains and no-slip boundary conditions has been solved for the second grade fluid equations, at least in a suitable regime of the parameters, see \cite{lopes2015approximation}.
The second-grade fluid equations are a model for viscoelastic fluids, with two parameters: $\eps>0$, corresponding to the elastic response, and $\nu>0$, corresponding to viscosity. For such fluids, assuming their density, $\rho$, being constant and equal to $1$, the stress tensor is given by 
\begin{equation*}
    T=-p^{SG,\eps}I+\nu A_1+\eps A_2-\eps A_1^2,
\end{equation*}
where \begin{align*}
    A_1=\frac{\nabla u^{SG,\eps}+\nabla (u^{SG,\eps})^T}{2},\\
    A_2=\partial_t A_1+A_1\nabla u^{SG,\eps}+(\nabla u^{SG,\eps})^T A_1,
\end{align*}
being $p^{SG,\eps}$ the pressure and $u^{SG,\eps}$ the velocity field.
Given this stress tensor, the equations of motion for an incompressible homogeneous fluid of grade 2 are given by
\begin{align}\label{second-grade system}
\begin{cases}
\partial_t v^{SG,\eps}=\nu \Delta u^{SG,\eps}-\operatorname{curl}(v^{SG,\eps})\times u^{SG,\eps}+\nabla p^{SG,\eps}+f^{SG,\eps}\\
\div u^{SG,\eps}=0\\
v^{SG,\eps}=u^{SG,\eps}-\eps\Delta u^{SG,\eps}\\
u^{SG,\eps}|_{\partial D}=0\\ 
u^{SG,\eps}(0)=u_0^{SG,\eps},
\end{cases}    
\end{align}
where $f^{SG,\eps}$ describes some external forces acting on the fluid, see \cite{dunn1974thermodynamics}, \cite{rivlin1997stress} for further details on the physics behind this system.
The analysis of the second-grade fluid equations started with \cite{cioranescu1984existence}, where some results, not restricted to the two-dimensional case, for global existence and uniqueness of solutions of the problem above have been shown. Setting, formally, $\eps=0$ the system above reduces to the Navier-Stokes system. Thus it can be seen as a generalization of the Navier-Stokes equations. Moreover, the convergence of the solution of the second-grade fluid equations to the solution of the Navier-Stokes equations has been shown rigorously in \cite{iftimie2002remarques}. 

\subsection{The Inviscid Limit in the Stochastic Framework}\label{subsec: stochastic inviscid intro}
In the last decades, stochastic forcing have been added to the fluid dynamic systems of the previous section. We refer to \cite[Chapter 5]{flandoli2023stochastic} for some justifications for the introduction of stochastic forcing terms in fluid dynamic models. After establishing the well-posedness of Navier-Stokes equations, Second-Grade fluid equations and Euler equations with Gaussian additive noise, see for example \cite{flandoli1995martingale}, \cite{razafimandimby2012strong}, \cite{bessaih19992} for some results in this direction, a natural question is trying to understand the validity of the inviscid limit in such stochastic models. According to \cite[Chapter 10]{kuksin2006randomly}, the relevant scaling of the parameters in order to study the inviscid limit is the following one
\begin{align}\label{NS introduction}
\begin{cases}
du^{NS,\epsilon}=(\epsilon \Delta u^{NS,\epsilon}-u^{NS,\epsilon}\cdot\nabla u^{NS,\epsilon}+\nabla p^{NS,\epsilon})dt+\sqrt{\eps}dW_t\\
\div u^{NS,\epsilon}=0\\
u^{NS,\epsilon}|_{\partial D}=0\\ 
u^{NS,\epsilon}(0)=u_0,
\end{cases}    
\end{align}

\begin{align}\label{second grade introduction}
\begin{cases}
dv^{SG,\epsilon}=(\nu \Delta u^{SG,\epsilon}-\operatorname{curl}(v^{SG,\epsilon})\times u^{SG,\epsilon}+\nabla p^{SG,\epsilon})dt+\sqrt{\eps}dW_t\\
\div u^{SG,\epsilon}=0\\
v^{SG,\epsilon}=u^{SG,\epsilon}-\eps\Delta u^{SG,\epsilon}\\
u^{SG,\epsilon}|_{\partial D}=0\\ 
u^{SG,\epsilon}(0)=u_0.
\end{cases}    
\end{align}
Difficulties analogous to those described in \autoref{subsec: inviscid limit intro} appear also in the stochastic framework, even considering different scaling of the parameters:
\begin{enumerate}
    \item  The validity of the inviscid limit in the case of Navier-Boundary conditions has been shown in \cite{cipriano2015inviscid}.
    \item  The validity of some conditioned results for the stochastic Navier-Stokes with no-slip boundary conditions has been shown in \cite{luongo2021inviscid}, \cite{wang2023kato}.
    \item The validity of the inviscid limit for the stochastic Second-Grade fluid equations with no-slip boundary conditions under suitable assumptions between $\nu$ and $\eps$ has been shown in \cite{luongo2022inviscid}.
\end{enumerate}
These results can be seen as a sort of law of large numbers for the stochastic systems above. A natural question then is that of understanding of the asymptotic behavior for the probability of large fluctuations away from the zero-noise and zero-viscosity limit. We are then interested in establishing Large Deviations principles for the aforementioned systems focusing, in case, on their relation with some form of Kato-type condition. At the same time, the validity of a Freidlin-Wentzell kind of LDP allows, at least in principle, to characterize rare phenomena arising in these systems in the small noise small viscosity regime, by translating these problems into an optimal control one. This might have several applications, and constitutes one of the main strength of the theory of Large Deviations (see for example the book by Freidlin and Wentzell \cite{FW}). 
In this paper we will study the validity of a Large Deviation Principle for the inviscid limit of Navier-Stokes equations and Second-Grade Fluid Equations with additive Gaussian noise in two-dimensional bounded domains and no-slip boundary, presenting, for the first system, a natural Kato-type condition that closely resembles the ones from classical conditioned results \cite{kato1984remarks},\cite{luongo2021inviscid}. According to the discussion in \autoref{subsec: inviscid limit intro} this is the critical case to analyze. Regarding the Large Deviation Principle for the inviscid limit of the Navier-Stokes equations, one technical issue that need to be addressed is the interplay between Kato-type conditions, i.e. some controls on the dissipation of the energy in the solutions of the stochastic Navier-Stokes equations within the boundary layer, and  large fluctuations away from the zero-noise and zero-viscosity limit.\\
Large Deviation for fluid dynamic models in 2D have been established in the case of Navier-Stokes with additive noise, see \cite{CHANGldpnseadd}, and multiplicative noise, see \cite{ldpnsemult}. While the first result is based on a technique developed by Freidlin and Wentzell, based on a discretization of the equation and the application of the so-called contraction principle, the second one resorts to the the Weak Convergence Approach developed in \cite{10.1214/07-AOP362}. While the Freidlin-Wentzell technique is best suited for equations with additive noise, the Weak Convergence Approach has proved to be much more flexible in many other situations. 
As an example, in \cite{:GaleatiLuo}, the authors proved a LDP for the convergence of the Euler equation with transport noise on the 2D torus to a deterministic Navier-Stokes system using the weak convergence approach. We adopt this approach as well, even if our equations have additive noise, as the vanishing of the viscosity together with the noise constitutes a technical issue that cannot be addressed via a classical contraction argument.\\ 
Finally, the validity of a Large Deviation Principle for the inviscid limit of the Navier-Stokes equations with periodic or free boundary conditions has been shown in \cite{bessaih2012large} using the weak convergence approach.
Similarly to other results with these kind of boundary conditions, the result of \cite{bessaih2012large} is based on the validity of the enstrophy equality, which allows to obtain stable estimates in the limit $\epsilon\rightarrow 0$ stronger than the one guaranteed by the energy equality. These relations are not available in the case of no-slip boundary conditions, due to the generation of vorticity close the boundary. Therefore, the introduction of some Kato-type hypothesis, see \cite{kato1984remarks}, is required in order to show the validity of the Large Deviation Principle, similarly to the validity of the inviscid limit.  On the contrary, as described in \autoref{subsec: inviscid limit intro}, there are fluid dynamic frameworks where the inviscid limit holds in the bounded domain without any assumption on the behavior of the solution in the boundary layer. This is the case of the second grade fluid equations with the scaling of the parameters introduced in \cite{lopes2015approximation} and the case of the Navier-Stokes equations in the open ball with forcing and initial conditions with radial symmetry. 
\subsection{Plan of the Paper}\label{subsec: plan of the paper}
The goal of this paper is to show the validity of the large deviations principle for the inviscid limit of relevant fluid dynamical systems in 2D bounded domain via the weak convergence approach introduced in \cite{10.1214/07-AOP362}. In \autoref{sec:Preliminaries} we will introduce some facts used repeatedly among the paper. In particular, in \autoref{subsec:LDP abstract}, we will recall several information about Large Deviation Principle, presenting its classical formulation and some equivalent formulations. In \autoref{subsec:well-known} we will present some well-known facts about the systems presented in this introduction that will play a role in the rest of the paper. In \autoref{subsec:main result} we will state our main theorems. The analysis of the validity of a Large Deviation Principle for the inviscid limit of the Navier-Stokes equations is the object of \autoref{sec: NS}. In particular we will prove \autoref{main thm LDP NS} in \autoref{sec:Kato assumption}, namely the validity of the Large Deviation Principle assuming a Kato condition. In \autoref{sec:radial symmetry} we will show that the Large Deviation Principle holds assuming initial conditions and forcing terms with radial symmetry. In \autoref{sec:second Grade} we will prove the validity of the Large Deviation Principle for the Second-Grade fluid equations under the scaling of the parameter introduced in \cite{lopes2015approximation}. Lastly we add some comments on the Kato condition assumed in this paper in \autoref{sec:Remark Kato}.
\section{Preliminaries and Main Results}\label{sec:Preliminaries}
\subsection{Large Deviations Principle}\label{subsec:LDP abstract}

We recall here the abstract framework of the weak convergence approach to Large Deviations developed in \cite{10.1214/07-AOP362}. We begin with an usual filtered probability space $(\Omega, \mathcal{F}, \mathcal{F}_t, \P),\ t\in[0,T]$. Let $\mathcal{H}$ be an Hilbert space and $\mathcal{Q}$ a trace-class operator on $\mathcal{H}$. We can endow the space $\mathcal{H}_0:=\mathcal{Q}^{1/2}\mathcal{H}$ with the metric induced by $\mathcal{Q}$, that is
$$\langle g, h\rangle_0= \langle \mathcal{Q}^{-1/2}g, \mathcal{Q}^{-1/2}h\rangle$$
which makes $\mathcal{H}_0$ a Hilbert space. The norm induced by this inner product will be denoted $\lVert \cdot\rVert_\0$. When $\mathcal{Q}$ is the covariance operator of a Wiener process $\{W_t\}_{t\in[0, T]}$ on $\mathcal{H}$, we call $\0$ the reproducing kernel Hilbert space of $W_t$, or simply RKHS. We also define the space \begin{align*}
  S^N:=S^N(\mathcal{H}_0):=\Big\{v\in L^2(0,T;\mathcal{H}_0):\quad \int_0^T \lVert v_s\rVert_\0^2 ds \leq N\Big\},  
\end{align*}
which makes a Polish space when endowed with the weak topology. 
We denote by $\mathcal{P}_2:=\mathcal{P}_2(\0)$ the space of $\0$-valued, $\mathcal{F}_t$-predictable and $\P$-a.s. square integrable processes. Next we define 
$$\mathcal{P}_2^N:=\{\phi\in \mathcal{P}_2: \ \phi(\omega)\in S^N\quad \mathbb{P}-a.s.\}$$
Let $\E$ and $\E_0$ be Polish spaces.
\begin{definition}
A function $I : \E \rightarrow [0,\infty]$ is called a rate function if for any $M <\infty$, the level set $\{f \in \E : I(f) \le M \}$ is a compact subset of $\E$. A family of rate functions $I_x$ on $\E$, parametrized by $x \in \E_0$, is said to have compact level sets on compacts if for all compact subsets $K$ of $\E_0$ and each $M < \infty$, $\cup_{x\in K}\{f \in \E : I_x(f)\le M \}$ is a compact subset of $\E$.
\end{definition}
Let us give now the definition of LDP in the original formulation by Varadhan (see \cite{varLDP})
\begin{definition}
We say that a Large Deviation Principle holds for a family $\mu^{\eps}$ of probability measures on a metric space $(\E, d)$ with rate function $I$ and speed $\eps$ if for every borel set $\Gamma$ of $\E$
\begin{equation}
    I(\ring \Gamma) \le \liminf_{\eps\rightarrow 0}\eps\log(\mu^{\eps}(\Gamma)) \le \limsup_{\eps\rightarrow 0}\eps\log(\mu^{\eps}(\Gamma))\le I(\bar\Gamma)
\end{equation}
where $I(A):=-\inf_{A} I$.
\end{definition}
This condition has been proved equivalent by Bryc in \cite{bryc} to the so called Laplace principle. Here, we state a uniform version of this principle, that is, we let $I$ and $\mu^\eps$ depend also on some parameter $x\in \E_0$.
\begin{definition}\textbf{(Uniform Laplace Principle)}
    Let $I_x$ be a family of rate functions on $E$ parameterized by $x\in \E_0$ and assume that this family has compact level sets on compacts. The family of random variables $\{X^{x, \eps}\}$ distributed according to $\mu^{x, \eps}$ are said to satisfy the Laplace principle on $\E$ with rate function $I_x$, uniformly on compacts, if for all compact subsets $K\subset \E_0$ and all bounded continuous functions $h$ mapping $\E$ into $\R$,
    \begin{equation}
        \lim_{\eps\rightarrow 0}\sup_{x\in K}\Big| \eps \log\EE_x\Big[\exp\big(-\eps^{-1}h(X^{x, \eps})\big) \Big] + \inf_{f\in \E}\{h(f) - I_x(f)\} \Big| =0
    \end{equation}
\end{definition}
We are interested in the case when the family of measures $\mu^\eps$ is given by the laws of some stochastic process $X^{x,\eps}$ solving some SPDE and driven by $\eps^{1/2} W_t$. In this case, we can often represent $X^{x, \eps}= \mathcal{G}^\eps(x, \eps^{1/2} W)$ for some measurable map $\mathcal{G}^\eps:\E_0 \times C([0,T];\mathcal{H}) \rightarrow \E$. In this setting, in \cite{10.1214/07-AOP362} the authors provided a handy criterium that allows to deduce the uniform Laplace principle. This is known as the \textit{weak convergence approach} to Large deviations. The criterium goes as follows.
\begin{hypothesis}\label{conditions weak convergence approach}
There exists a measurable map $\mathcal{G}^0 : \E_0 \times C([0,T];\mathcal{H})  \rightarrow \E$ such that:
\begin{enumerate}
\item For any $N<\infty$ and compact set $K\subset \E_0$ , $\Gamma _{K,N} :=\{ \mathcal{G}^0(x,\int_0^\cdot v_s ds):v\in S^N, \ x\in K \}$ is a compact subset of $\E$.
\item Consider $N<\infty$ and families $\{x^\eps\}\subset \E_0, \ \{u^\eps\}\subset \mathcal{P}_2^N$ such that, as $\eps\rightarrow 0$, $x^\eps\rightarrow x$ and $u^\eps$ converge in law to $u$ as $S^N$-valued random element, then  $\mathcal{G^\eps}\big(x^\eps, \eps^{1/2} W + \int_0^{\cdot} u^\eps_s ds\big)$ converges in law to $\mathcal{G}^0(x,\int_0^\cdot u_s ds)$ in the topology of $\E$.
\end{enumerate}
\end{hypothesis}
\begin{theorem}\label{weak convergence approach abstract thm}
  Let $X^{\eps,x}=\mathcal{G}^{\eps}(x,\eps^{1/2}W)$ and suppose \autoref{conditions weak convergence approach} holds. Define, for $x \in \E_0$ and $f \in \E$  
  \begin{align*}
     I_x(f) := \inf_{\{v\in L^2_t\0: \ f=\mathcal{G}^0(x,\int_0^\cdot v_s ds)\}} \int_0^T \|v_s\|^2_\0 ds 
  \end{align*}
with the convention that $\inf \emptyset = +\infty$. Assume that for all $f \in \E$, $x \rightarrow I_x(f)$ is a lower semicontinuous map from $\E_0$ to $[0,\infty]$. Then for all $x \in \E_0$, $f \rightarrow I_x(f)$ is a rate function on $E$ and the family ${I_x}x\in \E_0$ of rate functions has compact level sets on compacts. Furthermore, the family $\{X^{\eps,x}\}$ satisfies the Laplace principle on $\E$, with the rate functions $\{I_x\}$, uniformly on compact subsets of $\E_0$.
\end{theorem}

\subsection{Well-Known Facts on fluid dynamic models}\label{subsec:well-known}
Let us start this section introducing some general assumptions which will be always adopted under our analysis even if not recalled.
\begin{hypothesis}\label{General Hypothesis}
\begin{itemize}
    \item $0<T<+\infty$.
    \item $D$ is a bounded, smooth, simply connected domain.
    \item $\left(\Omega,\mathcal{F},\mathcal{F}_t,\mathbb{P}\right)$ is a filtered probability space such that $(\Omega, \mathcal{F},\mathbb{P})$ is a complete probability space, $(\mathcal{F}_t)_{t\in [0,T]}$ is a right continuous filtration and $\mathcal{F}_0$ contains every $\mathbb{P}$ null subset of $\Omega$.
%    \item $\{\eps_j\}_{j\in\mathbb{N}}$ and $\{\nu_j\}_{j\in\mathbb{N}}$ are two sequences of positive numbers converging to 0.
\end{itemize}
\end{hypothesis}
%In order to simplify the notation, we will consider $\eps,\ \nu>0$ in the following dropping the subscript $j$, having in mind that according to \autoref{General Hypothesis}, they are a countable family.\\
For square integrable semimartingales taking value in separable Hilbert spaces $U_1,\ U_2$ we will denote by
$[M, N]_t$ the quadratic covariation process. If $M, N$ take values in the same separable Hilbert space $U$ with orthonormal basis $u_i$, we will denote by $\langle\langle M,N\rangle\rangle_t=\sum_{i\in \mathbb N} [\langle M,u_i\rangle_U, \langle N,u_i\rangle_U]_t$.
For each $k\in \mathbb{N},\ 1\leq p\leq \infty$ we will denote by $L^p(D)$ and $W^{k,p}(D)$ the well-known Lebesgue and Sobolev spaces. We will denote by $C_{c}^{\infty}(D)$ the space of smooth functions with compact support and by $W^{k,p}_0(D)$ their closure with respect to the $W^{k,p}(D)$ topology.  If $p=2$, we will write $H^k(D)$ (resp. $H^k_0(D)$) instead of $W^{k,2}(D)$ (resp. $W^{k,2}_0(D)$). Let $X$ be a separable Hilbert space, denote by $L^p(\mathcal{F}_{t_0},X)$ the space of $p$ integrable random variables with values in $X$, measurable with respect to $\mathcal{F}_{t_0}$. We will denote by $L^p(0,T;X)$ the space of measurable functions from $[0,T]$ to $X$ such that 
\begin{align*}
    \lVert u\rVert_{L^p(0,T;X)}:=\left(\int_0^T \lVert u_t \rVert_X^p\ dt\right)^{1/p}<+\infty,\ \ 1\leq p<\infty
\end{align*}
and obvious generalization for $p=\infty.$ For any $r,\ p\geq 1$, we will denote by $L^p(\Omega,\mathcal{F},\mathbb{P};L^r(0,T;X))$ the space of processes with values in $X$ such that \begin{enumerate}
    \item $u(\cdot,t)$ is progressively measurable.
    \item $u(\omega,t)\in X$ for almost all $(\omega,t)$ and \begin{align*}
        \mathbb{E}\left[\lVert u(\omega,\cdot)\rVert_{L^r(0,T;X)}^p\right]<+\infty.
    \end{align*}
    Obvious generalizations for $p=\infty$ or $r=\infty$. 
\end{enumerate}
Lastly we will denote by $C_{\mathcal{F}}([0,T];H)$ the space of continuous adapted processes with values in $X$ such that 
\begin{align*}
    \mathbb{E}\left[\operatorname{sup}_{t\in  [0,T]}\lVert u_t\rVert_X^2\right]<+\infty.
\end{align*}

Set
\begin{align*}
  H=\{f \in L^2(D;\mathbb{R}^2),\ \div f=0,\ f\cdot n|_{\partial D}=0\},\ V=H_0^1(D;\mathbb{R}^2)\cap H,\ D(A)=H^2(D;\mathbb{R}^2)\cap V.  
\end{align*}
Moreover we introduce the vector space
\begin{align*}
  W=\{u\in V:\ \operatorname{curl}(u-\eps\Delta u)\in L^2(D;\mathbb{R}^2) \}  
\end{align*}
 with norm $\lVert{u}\rVert_W^2=\lVert u\rVert^2+\eps\lVert \nabla u\rVert_{L^2(D;\mathbb{R}^2)}^2+\lVert \operatorname{curl}(u-\Delta u)\rVert_{L^2(D)}^2.$ It is well-known, see for example \cite{cioranescu1984existence}, that we can identify $W$ with the space 
 \begin{align*}
  \hat{W}=\{u\in H^3(D;\mathbb{R}^2)\cap V \}.   
 \end{align*}
Moreover there exists a constant such that \begin{align}\label{equivalence H3-W}
    \lVert u\rVert_{H^3}^2\leq C\left(\lVert u\rVert^2+\lVert \nabla u\rVert_{L^2(D;\mathbb{R}^2)}^2+\lVert \operatorname{curl}(u-\Delta u)\rVert^2_{L^2(D)} \right).
\end{align} We denote by $\langle\cdot,\cdot\rangle$ and $\lVert\cdot\rVert$ the inner product and the norm in $H$ respectively. With some abuse of notation we will denote also the inner product and the norm in $L^2$ with the same symbols. Other norms and scalar products will be denoted with the proper subscript. On $V$ we introduce the norm $\lVert u\rVert_V^2=\lVert u\rVert^2+\eps\lVert \nabla u\rVert_{L^2(D;\mathbb{R}^2)}^2.$ We will shortly denote by $\lVert u\rVert_*=\lVert \operatorname{curl}(u-\eps\Delta u)\rVert_{L^2(D)}.$ Obviously the following inequality holds for $u\in V$, where $C_p$ is the Poincarè constant associated to $D$,
\begin{align}\label{equivalence H1-V}
    \frac{\lVert u\rVert_V^2}{\eps+C_p^2}\leq \lVert \nabla u\rVert_{L^2(D;\mathbb{R}^2)}^2\leq \frac{\lVert u\rVert_V^2}{\eps}
\end{align}
Denote by $P$ the linear projector of $L^2\left(D;\mathbb{R}^2\right)$ on $H$ and define the unbounded linear operator $A:D(A)\subseteq H\rightarrow H$ by the identity \begin{align}\label{definition of A}
    \langle A v, w\rangle=\langle \Delta v, w \rangle_{L^2(D;\mathbb{R}^2)}
\end{align}
for all $v \in D(A),\ w \in H$. $A$ will be called the Stokes operator. It is well-known (see for example \cite{temam2001navier}) that $A$ is self-adjoint, generates an analytic semigroup of negative type on $H$ and moreover $V\sim D\left(\left(-A)^{1/2}\right)\right).$ We will denote by $e^{\eps At}$ the strongly continuous semigroup on $H$ generated by $\eps A$.
Denote by $\mathbb{L}^{4}$ the space $L^{4}\left(  D,\mathbb{R}^{2}\right)
\cap H$, with the usual topology of $L^{4}\left(  D,\mathbb{R}^{2}\right)  $.
Define the trilinear, continuous form $b:\mathbb{L}^{4}\times V\times\mathbb{L}%
^{4}\rightarrow\mathbb{R}$ as%
\begin{align}\label{definition of b}
b\left(  u,v,w\right)  =\langle u, P(\nabla v w)\rangle.
\end{align}
We will use also some properties of the projection operator $P$ and the solution map of the Stokes operator. We refer to \cite{temam2001navier} for the proof of the lemmas below.
\begin{lemma}\label{lemma projection}
The restriction of the projection operator $P:L^2(D;\mathbb{R}^2)\rightarrow H$ to $H^r(D;\mathbb{R}^2)$ is a continuous and linear map between $H^r(D;\mathbb{R}^2)$ and itself.
\end{lemma}
\begin{lemma}\label{eigenfunction Stokes}
The injection of $V$ in $H$ is compact. Thus there exists a sequence ${e}_i$ of elements of $H$ which forms an orthonormal basis in $H$ and an orthogonal basis in $V$. This sequence verifies \begin{align*}
    -A{e}_i={\lambda}_i{e}_i
\end{align*}
where ${\lambda}_{i+1}>{\lambda}_{i}>0,\ i=1,2,\dots$. Moreover ${\lambda}_i\rightarrow +\infty$. Lastly ${e}_i\in C^{\infty}(\overline{D};\mathbb{R}^2)$  under our assumptions on $D$ 
\end{lemma}
The tools introduced above are the standard ingredients in order to deal with the Navier-Stokes equations. We need to recall some other facts in order to treat Second-Grade fluid equations.  We refer to \cite{cioranescu1984existence},\cite{razafimandimby2010weak}, \cite{razafimandimby2012strong}, \cite{galdi2011introduction} for the proof of the various statements.
\begin{lemma}\label{nonlinearity}
For any smooth, divergence free $\phi,\ v,\ w$ the following relation holds \begin{align}\label{equvalence hatB and b}
    \langle \operatorname{curl}\phi \times v, w\rangle_{L^2}=b(v,\phi,w)-b(w,\phi,v).
\end{align}
Moreover for $u,\ v,\ w$ the following inequalities hold \begin{align}
    \lvert \langle \operatorname{curl}(u-\eps\Delta u) \times v, w\rangle_{L^2}\rvert\leq C \lVert u\rVert_{H^3}\lVert v\rVert_V \lVert w\rVert_W \label{inequality trilinear 1}\\
    \lvert \langle \operatorname{curl}(u-\eps\Delta u) \times u, w\rangle_{L^2}\rvert\leq C \lVert u\rVert^2_V \lVert w\rVert_W\label{inequality trilinear 2}
\end{align}
Therefore there exists a bilinear operator $\hat{B}:W\times V\rightarrow W^*$ such that \begin{align}\label{definition hatB}
    \langle \hat{B}(u,v),w\rangle_{W^*,W}=\langle P(\operatorname{curl}(u-\eps\Delta u)\times v),w \rangle
\end{align}
which satisfies for $u\in V,\ v\in W$
\begin{align}
    \lVert \hat{B}(v,u)\rVert_{W^*}\leq C \lVert u\rVert_V\lVert v\rVert_W \label{inequality bilinear 1}\\ 
    \lVert \hat{B}(u,u)\rVert_{W^*}\leq C \lVert u\rVert_V^2\label{inequality bilinear 2}.
\end{align}
Lastly, for $u\in W,\ v\in V, \ w \in W$
\begin{align}\label{antisimmetry hatB}
    \langle \hat{B}(u,v),w\rangle_{W^*,W}=-\langle \hat{B}(u,w),v\rangle_{W^*,W}.
\end{align}
\end{lemma}
\begin{theorem}Each function $f\in H^2(D)$ satisfies the following inequality:
\begin{align}\label{interpolation estimate}
\lVert f\rVert_{H^1}\leq C   \lVert f\rVert_{L^2}^{1/2}\lVert f\rVert_{H^2}^{1/2}.
\end{align}
\end{theorem}
Now we are ready to introduce some assumptions on the stochastic part of systems \eqref{NS introduction}, \eqref{second grade introduction}.
\begin{hypothesis}\label{hypothesis noise}
$W_t=\sum_{k\in K}\sigma_k W^k_t $ where
\begin{itemize}
    \item $K$ is a (possibly countable) set of indexes, $\gamma\geq 2$. 
    \item $\sigma_k\in D((-A)^{\gamma})$ satisfying \begin{align*}
        \sum_{k\in K}\lVert \sigma_k\rVert_{D((-A)^{\gamma})}^2<+\infty.
    \end{align*}
    \item $\{W^k_t\}_{k\in K}$ is a sequence of real, independent Brownian motions adapted to $\mathcal{F}_t$.
\end{itemize}
\end{hypothesis}
We denote by $H_0$ the RKHS associated to $W_t$.
\begin{remark}\label{remark on noise 1}
Previous assumptions on the noise implies in particular that $H_0\hookrightarrow D((-A)^\gamma)$ and that $W$ is a process with continuous paths vith values in $D((-A)^{\gamma})$. Since $\lambda_i\sim Ci,$ see \cite{ilyin2008spectrum}, a simple example of noise satisfying \autoref{hypothesis noise} is $W_t=(-A)^{-\gamma-1/2-\delta}W_t^{H}$, $\delta>0$ and $W_t^{H}$ being the cylindrical Wiener process on $H$. With this particular choice of the coefficients $\sigma_k$, $H_0=D((-A)^{\gamma+1/2+\delta})$.
\end{remark}
Since we are going to prove the validity of the Large Deviation Principle via the weak convergence approach, we will need to analyze the well-posedness of some partial differential equations, slightly more general than \eqref{NS introduction}, \eqref{second grade introduction}. Therefore, let $\beta\geq 0$ and $f\in \mathcal{P}_2^N,\ N\geq 0$ arbitrary we consider the stochastic partial differential equations below  
\begin{align}\label{NS with forcing}
\begin{cases}
du^{NS,\epsilon}=(\epsilon \Delta u^{NS,\epsilon}-u^{NS,\epsilon}\cdot\nabla u^{NS,\epsilon}+\nabla p^{NS,\epsilon}+f)dt+\sqrt{\beta}dW_t\\
\div u^{NS,\epsilon}=0\\
u^{NS,\epsilon}|_{\partial D}=0\\ 
u^{NS,\epsilon}(0)=u_0,
\end{cases}    
\end{align}

\begin{align}\label{second grade with forcing}
\begin{cases}
dv^{SG,\epsilon}=(\nu \Delta u^{SG,\epsilon}-\operatorname{curl}(v^{SG,\epsilon})\times u^{SG,\epsilon}+\nabla p^{SG,\epsilon}+f)dt+\sqrt{\beta}dW_t\\
\div u^{SG,\epsilon}=0\\
v^{SG,\epsilon}=u^{SG,\epsilon}-\eps\Delta u^{SG,\epsilon}\\
u^{SG,\epsilon}|_{\partial D}=0\\ 
u^{SG,\epsilon}(0)=u_0.
\end{cases}    
\end{align}
\begin{definition}\label{weak solution NS forcing}
 A stochastic process with continuous trajectories with values in $H$
is a weak solution of equation \eqref{NS with forcing} if
\begin{align*}
u^{NS,\eps} \in C_{\mathcal{F}}([0,T];H)\cap L^2(\Omega,\mathcal{F},\mathbb{P};L^{2}(0,T;V))    
\end{align*}
and $\mathbb{P}-a.s.$ for every $t\in [0,T]$ and $\phi \in D(A)$  we have
\begin{align*}
    &\langle u^{NS,\eps}_t-u_0,\phi\rangle+\int_0^t \eps \langle \nabla u^{NS,\eps}_s,\nabla \phi\rangle_{L^2(D;\mathbb{R}^2)} =\int_0^t b(u^{NS,\eps}_s,\phi, u^{NS,\eps}_s)ds+\int_0^t \langle f_s,\phi \rangle ds +\sqrt{\beta}\langle W_t,\phi\rangle.
\end{align*}       
\end{definition}
\begin{definition}\label{weak solution second grade forcing}
A stochastic process with weakly continuous trajectories with values in $W$
is a weak solution of equation \eqref{second grade with forcing} if
\begin{align*}
u^{SG,\eps} \in L^2(\Omega,\mathcal{F},\mathbb{P};L^{\infty}(0,T;W))    
\end{align*}
and $\mathbb{P}-a.s.$ for every $t\in [0,T]$ and $\phi \in W$  we have
\begin{align*}
    &\langle u^{SG,\eps}_t-u_0,\phi\rangle_V+\int_0^t \nu \langle \nabla u^{SG,\eps}_s,\nabla \phi\rangle_{L^2(D;\mathbb{R}^2)}+ \langle \operatorname{curl}(u^{SG,\eps}_s-\eps\Delta u^{SG,\eps}_s)\times u^{SG,\eps}_s, \phi\rangle_{L^2(D)} ds\\ & =\int_0^t \langle f_s,\phi \rangle ds +\sqrt{\beta}\langle W_t,\phi\rangle.
\end{align*}    
\end{definition}
The well-posedness of \eqref{NS with forcing} (resp. \eqref{second grade with forcing}) in the sense of \autoref{weak solution NS forcing} (resp. \autoref{weak solution second grade forcing}) is guaranteed by \autoref{thm well-posed ns forcing} below, see \cite{flandoli2023stochastic} (resp. \autoref{thm well-posed sf forcing}, see \cite{razafimandimby2012strong}, \cite[Section 6]{luongo2022inviscid}).
\begin{theorem}\label{thm well-posed ns forcing}
For each $u_0\in H^3(D;\mathbb{R}^2)\cap H$ there exists a unique weak solution of \eqref{NS with forcing} in the sense of \autoref{weak solution NS forcing}. Moreover the following relation holds true \begin{align}\label{ito NS}
    \lVert u^{NS,\eps}_t\rVert^2+2\eps \int_0^t \lVert \nabla u^{NS,\eps}_s\rVert^2_{L^2}ds&=\lVert u_0\rVert^2+t\beta\sum_{k\in K}\lVert \sigma_k\rVert^2+2\sqrt{\beta}\int_0^t \langle u^{NS,\eps},dW_s\rangle+2\int_0^t\langle f_s, u^{NS,\eps}\rangle ds\quad \mathbb{P}-a.s.
\end{align}  
\end{theorem}
\begin{theorem}\label{thm well-posed sf forcing}
For each $u_0\in W$ there exists a unique weak solution of \eqref{second grade with forcing} in the sense of \autoref{weak solution second grade forcing}. Moreover $u^{SG,\eps}$ has continuous paths with values in $V$ and it holds
\begin{align}\label{ITO SG 1}
     \lVert u^{SG,\eps}_t\rVert_V^2+2\nu \int_0^t \lVert \nabla u^{SG,\eps}_s\rVert^2_{L^2}ds&=\lVert u_0\rVert_V^2+t\beta\sum_{k\in K}\lVert (I-\eps A)^{-1/2}\sigma_k\rVert^2\notag\\ &+2\sqrt{\beta}\int_0^t \langle u^{SG,\eps},dW_s\rangle+2\int_0^t\langle f_s, u^{SG,\eps}\rangle ds\quad \mathbb{P}-a.s.
\end{align}
Calling $q^{SG,\eps}=\operatorname{curl}(u^{SG,\eps}-\eps \Delta u^{SG,\eps}),\ s_k=\operatorname{curl}\sigma_k$ it holds
\begin{align}\label{ITO SG 2}
 \lVert q^{SG,\eps}_t\rVert_{L^2}^2 & =\lVert u_0\rVert_*^2-\frac{2\nu}{\eps}\int_0^t\langle q_s^{SG,\eps}-\operatorname{curl}u_s^{SG,\eps},q_s^{SG,\eps}\rangle ds+t\beta\sum_{k\in K}\lVert s_k\rVert^2\notag\\ &+2\sqrt{\beta}\sum_{k\in K}\int_0^t \langle s_k,q_s\rangle dW^k_s+2\int_0^t \langle \operatorname{curl}f_s,q_s\rangle ds\quad \mathbb{P}-a.s.
\end{align}
\end{theorem}
Lastly we need to recall some results about the well-posedness of Euler equations with forcing term $f\in \mathcal{P}_2^N$, namely solutions of \begin{align}\label{euler forcing}
\begin{cases}
        \partial_t \overline{u}+\overline{u}\cdot\nabla\overline{u}=\nabla \overline{p}+f\\
        \div \overline{u}=0\\
        \overline{u}\cdot n|_{\partial D}=0\\ 
        \overline{u}(0)=u_0.
\end{cases}
\end{align}
\begin{definition}\label{weak euler}
A stochastic process with continuous trajectories with values in $H$
is a weak solution of equation \eqref{euler forcing} if
 $\mathbb{P}-a.s.$ for every $t\in [0,T]$ and $\phi \in C^{\infty}_c(D;\mathbb{R}^2)$  we have
\begin{align*}
    &\langle \overline{u}_t-u_0,\phi\rangle =\int_0^t b(\overline{u}_s,\phi, \overline{u}_s)ds-\int_0^t \langle f_s,\phi \rangle ds.
\end{align*}      
\end{definition}
The well-posedness of \eqref{euler forcing} in regular spaces is a classical result, see for example \cite{bessaih19992},\cite{bessaih2020invariant},\cite{10.1214/12-AOP773}.
\begin{theorem}\label{thm well posed euler forcing}
    For each $u_0\in W $ there exists a unique weak solution of \eqref{euler forcing} with trajectories in $C([0,T];W^{2,4}(D;\mathbb{R}^2))$. Moreover
    \begin{align}\label{Ito Euler }
        \lVert\overline{u}_t\rVert^2=\lVert u_0\rVert^2+2\int_0^t \langle f_s,\overline{u}_s\rangle ds\quad \mathbb{P}-a.s.
    \end{align}
    \begin{align}\label{Energy Euler }
        \operatorname{sup}_{t\in [0,T]}\lVert\overline{u}_t\rVert_{W^{2,4}}\leq C(\lVert u_0\rVert_{W^{2,4}},N)\quad \mathbb{P}-a.s.
    \end{align}
\end{theorem}
\begin{remark}\label{remark on noise 2}
   The well-posedness of \autoref{NS with forcing}, \autoref{second grade with forcing} holds under weaker assumptions on the noise than \autoref{hypothesis noise}. We need to assume a noise so regular in order to guarantee that there exists a unique solution of \eqref{euler forcing} which belongs to $C([0,T];W^{2,4}(D;\mathbb{R}^2))\cap C([0,T];H)$.
\end{remark}
\subsection{Main Results}\label{subsec:main result}
As stated in \autoref{sec:introduction}, our goal is to prove a Large Deviation Principle via the weak convergence approach introduced in \autoref{subsec:LDP abstract}. Therefore we need to introduce some maps $\mathcal{G}^{NS,\eps},\ \mathcal{G}^{SG,\eps},\ \mathcal{G}^{0}.$ Following the notation of \autoref{subsec:LDP abstract}, let
\begin{align*}
   \mathcal{E}_0^{NS}:=H^3(D;\mathbb{R}^2)\cap H,\quad \mathcal{E}_0^{SG}:=W,\quad \mathcal{E}:=C([0,T];H).
\end{align*}
According to \autoref{thm well posed euler forcing} we can introduce the measurable map
\begin{align*}
    \mathcal{G}^{NS,0}:\mathcal{E}_0^{NS}\times C([0,T];H)\rightarrow \mathcal{E}\quad \left(\textit{resp }   \mathcal{G}^{SG,0}:\mathcal{E}_0^{SG}\times C([0,T];H)\rightarrow \mathcal{E}\right)
\end{align*}
which associates to each $u_0\in\mathcal{E}_0^{NS} $ (resp. $u_0\in\mathcal{E}_0^{SG} $)  and $\int_0^\cdot f_s ds,\ f \in L^2(0,T;H_0)$ the unique regular solution of \eqref{euler forcing} with initial condition $u_0$ and forcing term $f$ guaranteed by \autoref{thm well posed euler forcing}, $0$ otherwise.
Analogously thanks to \autoref{thm well-posed ns forcing} (resp. \autoref{thm well-posed sf forcing}) we can introduce the measurable map
\begin{align*}
    \mathcal{G}^{NS,\eps}:\mathcal{E}_0^{NS}\times C([0,T];H)\rightarrow \mathcal{E}\quad \left(\textit{resp }   \mathcal{G}^{SG,\eps}:\mathcal{E}_0^{SG}\times C([0,T];H)\rightarrow \mathcal{E}\right)
\end{align*}
such that for each $u_0\in\mathcal{E}_0^{NS} $ (resp. $u_0\in\mathcal{E}_0^{SG} $), $\mathcal{G}^{NS,\eps}(u_0,\sqrt{\eps}W_\cdot)$ (resp. $\mathcal{G}^{SG,\eps}(u_0,\sqrt{\eps}W_\cdot)$ is the unique weak solution of \eqref{NS with forcing} (resp. \eqref{second grade with forcing}) with $\beta=\eps$, initial condition $u_0$ and null forcing term guaranteed by \autoref{thm well-posed ns forcing} (resp. \autoref{thm well-posed sf forcing}). More in general, it follows that, if $f\in \mathcal{P}_2^N$, $\mathcal{G}^{NS,\eps}(u_0,\sqrt{\eps}W_\cdot+\int_0^\cdot f_s ds )$ (resp. $\mathcal{G}^{SG,\eps}(u_0,\sqrt{\eps}W_\cdot+\int_0^\cdot f_s ds )$) is the unique solution of \eqref{NS with forcing} (resp. \eqref{second grade with forcing}) $\beta=\eps$, initial condition $u_0$ and forcing term $f$.
When dealing with the inviscid limit for Navier-Stokes equations and no-slip boundary conditions one can choose either to assume a Kato-type hypothesis or to require strong assumptions on the regularity of the domain, initial conditions and forcing term. We will follow both these lines. In the following, given $c>0$, we will denote $\Gamma_{c\eps}=\{x\in D:\quad d(x,\partial D)\leq c\eps\}$.
\begin{hypothesis}[Strong Kato Hypothesis]\label{strong kato hp}
 For each $N\in \mathbb{N}$, $u_0^{\eps},\ u_0\in \mathcal{E}_0^{NS}$ and $f^{\eps},\ f\in \mathcal{P}_2^N$ such that $u_0^{\eps}\rightarrow u_0$ in $\mathcal{E}_0^{NS}$ and $f^{\eps}\rightarrow_{\mathcal{L}} f$ in $S^N$, if $(\Omega,\mathcal{F},\mathcal{F}_t,\mathbb{P})$  is a filtered probability space where all $f^{\eps}$, $f$ are defined together and $f^{\eps}\rightarrow f \ \mathbb{P}-a.s.$ in $S^N$, then,  it exists $c>0$ such that for every $\delta>0$
\begin{align*}
    \mathbb{P}\left(\eps\int_0^T\left\lVert \nabla \mathcal{G}^{NS,\eps}\left(u_0^{\eps}, \sqrt{\eps}W_\cdot+\int_0^\cdot f^{\epsilon}_s ds \right) \right\rVert_{L^2(\Gamma_{c\eps})}^2 ds>\delta\right)\rightarrow 0.
\end{align*}    
\end{hypothesis}
We will come back on the meaning of \autoref{strong kato hp} and its relation with a more classical version of the Kato Hypothesis, i.e. non depending from $f^{\eps}$, in \autoref{sec:Remark Kato}.
\begin{remark}
    By Skorokhod's representation theorem, given $f^{\eps},\ f\in \mathcal{P}_2^N$ such that $f^{\eps}\rightarrow_{\mathcal{L}} f$ in $S^N$ there exists at least a filtered probability space $(\Omega,\mathcal{F},\mathcal{F}_t,\mathbb{P})$ where all $f^{\eps}$, $f$ are defined together and $f^{\eps}\rightarrow f \ \mathbb{P}-a.s.$ in $S^N$.
\end{remark}
Now we are ready to state our main result on the validity of a Large Deviation Principle under \autoref{strong kato hp}.
\begin{theorem}\label{main thm LDP NS}
Assuming \autoref{strong kato hp}, the solutions $\{u^{NS,\eps}=\mathcal{G}^{NS,\eps}(u_0,\sqrt{\eps}W_\cdot)\}_{u_0\in \mathcal{E}_0^{NS}}$ satisfy the uniform Laplace principle with the rate function \begin{align*}
    I_{u_0}^{NS}(v)=\operatorname{inf}_{f\in L^2(0,T;H_0):\ v=\mathcal{G}^{NS,0}(u_0,\int_0^\cdot f_s ds )}\frac{1}{2}\int_0^T \| f_s\|_{H_0}^2 ds 
\end{align*}
where $u_0\in \mathcal{E}_0^{NS},\ v\in C([0,T];H)$.
\end{theorem}
As pointed out in \cite{lopes2015approximation}, \cite{luongo2022inviscid} in order to obtain an unconditioned result for the Second-Grade fluid equations, we cannot take any scaling of $\nu\rightarrow 0$ but it is necessary to assume:
\begin{hypothesis}\label{Second Grade Fluids assumption}
 $\nu=O(\eps)$.   
\end{hypothesis}
Now we can state our main result on the Second Grade Fluid equations.
\begin{theorem}\label{main thm LDP 2GF}
Assuming \autoref{Second Grade Fluids assumption}, the solutions $\{u^{SG,\eps}=\mathcal{G}^{SG,\eps}(u_0,\sqrt{\eps}W_\cdot)\}_{u_0\in \mathcal{E}_0^{SG}}$ satisfy the uniform Laplace principle with the rate function \begin{align*}
    I^{SG}_{u_0}(v)=\frac{1}{2}\operatorname{inf}_{f\in L^2(0,T;H_0):\ v=\mathcal{G}^{SG,0}(u_0,\int_0^\cdot f_s ds )}\int_0^T \lVert f_s\rVert_{H_0}^2 ds 
\end{align*}
where $u_0\in \mathcal{E}_0^{SG},\ v\in C([0,T];H)$.
\end{theorem}
Lastly we want to consider the case of fluids with radial symmetry. In such case the inviscid limit in general holds without any assumptions on the behavior of the fluid in the boundary layer as observed in \cite{lopes2008vanishing}. Therefore, calling $B$ the open ball in $\mathbb{R}^2$, centered in $0$ with radius $1$, we introduce 
\begin{align*}
   \mathcal{H}^{RS}=\overline{\left\{\frac{x^{\perp}}{\lvert x\rvert }\bar{u}(\lvert x\rvert),\quad \bar{u}\in C_c^{\infty}(0,1)\right\}}^{D((-A)^{\gamma})}
\end{align*}
endowed with the $D((-A)^{\gamma})$ norm and 
\begin{align*}
   \mathcal{E}_0^{RS}=\mathcal{E}_0^{NS}\cap \left\{ u=\frac{x^{\perp}}{\lvert x\rvert }\bar{u}(\lvert x\rvert),\quad \bar{u}\in L^2(0,1)\right\}
\end{align*}
endowed with the $H^3$ norm. As above we need to introduce a particular forced Navier-Stokes systems: \begin{align}\label{Symmetry with forcing}
\begin{cases}
du^{RS,\epsilon}=(\epsilon \Delta u^{RS,\epsilon}-u^{RS,\epsilon}\cdot\nabla u^{RS,\epsilon}+\nabla p^{RS,\epsilon}+f)dt+\sqrt{\eps}dW^{RS}_t\\
\div u^{RS,\epsilon}=0\\
u^{RS,\epsilon}|_{\partial D}=0\\ 
u^{RS,\epsilon}(0)=u_0,
\end{cases}    
\end{align}
Now can introduce the assumptions in order to deal the case with radial symmetry and study the Large Deviation Principle in this framework:
\begin{hypothesis}\label{hypothesis symmetry}
$D=B$, $W^{RS}_t=\sum_{k\in K}\sigma_k W^k_t $ where
\begin{itemize}
    \item $K$ is a (possibly countable) set of indexes, $\gamma\geq 2$. 
    \item $\sigma_k\in \mathcal{H}^{RS}$ satisfying \begin{align*}
        \sum_{k\in K}\lVert \sigma_k\rVert_{D((-A)^{\gamma})}^2<+\infty.
    \end{align*}
    \item $\{W^k_t\}_{k\in K}$ is a sequence of real, independent Brownian motions adapted to $\mathcal{F}_t$.
\end{itemize}
\end{hypothesis}
We denote by $H_0^{RS}$ the RKHS associated to $W_t^{RS}$.\\
Since \autoref{thm well-posed ns forcing}, \autoref{euler forcing} continue to hold considering $u_0\in \mathcal{E}_0^{RS},\ f\in \mathcal{H}^{RS}$ and assuminh \autoref{hypothesis symmetry}, we can define the measurable maps $\mathcal{G}^{RS,\eps}$ and $\mathcal{G}^{RS,0}$ as above for  $\mathcal{G}^{NS,\eps}$ and $\mathcal{G}^{NS,0}$ considering $\mathcal{E}_0^{RS}$ instead of $\mathcal{E}_0^{NS}.$

\begin{theorem}\label{main thm LDP NS radial symmetry}
Assuming \autoref{hypothesis symmetry}, the solutions $\{u^{RS,\eps}=\mathcal{G}^{RS,\eps}(u_0,\sqrt{\eps}W^{RS}_\cdot)\}_{u_0\in \mathcal{E}_0^{RS}}$ satisfy the uniform Laplace principle with the rate function \begin{align*}
    I_{u_0}^{RS}(v)=\frac{1}{2}\operatorname{inf}_{f\in L^2(0,T;H_0):\ v=\mathcal{G}^{RS,0}(u_0,\int_0^\cdot f_s ds )}\int_0^T \lVert f_s\rVert_\0^2 ds 
\end{align*}
where $u_0\in \mathcal{E}_0^{RS},\ v\in C([0,T];H)$.
\end{theorem}

\begin{remark}\label{indifference to forcing}
\autoref{main thm LDP NS}, \autoref{main thm LDP 2GF} and \autoref{main thm LDP NS radial symmetry} continue to hold also if we add a deterministic forcing term $g$ in $L^2(0,T;H_0)$ or $L^2(0,T;H_0^{RS})$ in equations \eqref{NS introduction}, \eqref{second grade introduction}, up to re-defining the maps $\G^\eps$ and $\G^0$ accordingly. Indeed the computations below can be easily adapted to this framework. Moreover, it is enough to assume the validity of \autoref{strong kato hp} for equation \eqref{NS introduction} without any forcing $g$ in order to prove the validity of \autoref{main thm LDP NS} also if we add the forcing term $g$.
\end{remark}

\begin{remark}\label{simplification functionals}
In the framework of \autoref{remark on noise 1}, 
$I_{u_0}^{NS}(v)$ reduces to \begin{align*}
    I_{u_0}^{NS}(v)=\frac{1}{2}\int_0^T \lVert (-A)^{\gamma+1/2+\delta} \left(\partial_s v_s-P(v_s\cdot\nabla v_s)\right)\rVert^2 ds.
\end{align*}
Similarly for the other functional.
\end{remark}
We conclude this section with few notations that will be adopted:
by $C$ we will denote several constant independent from $\eps$, $\nu$, perhaps changing value line by line. If we want to keep track of the dependence of $C$ from some parameter $\rho$ we will use the symbol $C(\rho)$. In order to simplify the notation we will denote Sobolev spaces by $W^{s,p}$, forgetting domain and range and use Einstein summation convention.

\section{Navier-Stokes}\label{sec: NS}
\subsection{Proof of \autoref{main thm LDP NS}.}\label{sec:Kato assumption}
\subsubsection{Condition 1}\label{subsubsec:cond1 Kato}
Let us fix $N>0,\ K$ a compact subset of $\mathcal{E}_0^{NS}$, we want to show that the set
\begin{align*}
 K_N=\{\mathcal{G}^{NS,0}(x_0,\int_0^\cdot f_s ds)\,\quad v\in S^N, x_0\in K\} \stackrel{c}{\hookrightarrow}\mathcal{E}.
\end{align*}
Therefore let us fix two sequences $\{x_{0}^n\}_{n\in \mathbb{N}}\subset K,\ \{f^n\}_{n\in \mathbb{N}}\subset S^N.$ Since $K$ is compact subset of $\mathcal{E}_0^{NS},\ \lVert v^n\rVert_{L^2(0,T;H_0)}^2\leq N $ we can find a subsequence $\{n_k\}_{k\in\mathbb{N}},\ x\in K,\ f\in S^N$ such that $x_0^{n_k}\rightarrow x_0$ in $\mathcal{E}^{NS}_0,\ f^{n_k}\ \rightharpoonup f  $ in $L^2(0,T;H_0)$. Let $u^{n_k}:=\mathcal{G}^{NS,0}(x_0^{n_k},\int_0^\cdot f^{n_k}_s ds )$  (resp. $u:=\mathcal{G}^{NS,0}(x_0,\int_0^\cdot f_s ds )$). According to \autoref{thm well posed euler forcing}, $u^{n_k}$ (resp $u$) is the unique regular weak solution of \eqref{euler forcing} with initial condition $x_0^{n_k}$ (resp. $x_0$) and forcing term $f^{n_k}$ (resp. $f$). Our goal is to show that $u^{n_k}\rightarrow u$ in $\mathcal{E}$. Fix $\theta>0$ arbitrarily small and define $F^{n_k}=\int_0^\cdot f^{n_k}_s ds,\ F=\int_0^\cdot f_s ds$. By hypothesis \autoref{hypothesis noise}, $H_0\hookrightarrow D((-A)^2)$. This implies, see for example \cite[Proposition 26]{flandoli2023stochastic} that \begin{align}\label{convergence integral forcing cond1}
    F^{n_k}\rightarrow F \quad\textit{in }C([0,T];D((-A)^{2-\theta})).
\end{align} Obviously
\begin{align}\label{uniform bound integral forcing cond1}
\operatorname{sup}_{k\geq 1}\lVert F^{n_k}\rVert_{C([0,T];D((-A)^2))}+\lVert F\rVert_{C([0,T];D((-A)^2))}\leq C(N).  
\end{align}
Lastly, since $x_0^{n_k}\rightarrow x_0$ in $\mathcal{E}_0^{NS}$, from \eqref{Energy Euler } we can find a constant $C=C(N,u_0)$ only depending from $N$ and $\lVert x\rVert_{\mathcal{E}_0^{NS}}$ such that 
\begin{align}\label{uniform energy W24 euler}
    \operatorname{sup}_{k\geq 1}\lVert {u}^{n_k}\rVert_{C([0,T];W^{2,4})}+\lVert {u}\rVert_{C([0,T];W^{2,4})}\leq C(\lVert x_0\rVert_{W^{2,4}},N).
\end{align}
We introduce 
\begin{align*}
    & v^{n_k}_t=u^{n_k}_t-F^{n_k}_t,\quad v_t=u_t-F_t. 
\end{align*}
By triangle inequality $v^{n_k},\ v$ satisfy equation \eqref{uniform energy W24 euler}, too.
Since $F^{n_k}\rightarrow F$ in $\mathcal{E}$, it is enough to show that $v^{n_k}\rightarrow v$ in $\mathcal{E}$ in order to prove the validity of Condition 1 in \autoref{conditions weak convergence approach}.
This is the aim of \autoref{lemma convergence euler} below. We will follow the idea introduced in \cite{YUDOVICH19631407} to show uniqueness of the solutions with bounded vorticity of the Euler equations. However, in order to prove the continuous dependence from the data we exploit the higher regularity and the uniform bounds guaranteed by relation \eqref{uniform energy W24 euler}.
\begin{lemma}\label{lemma convergence euler}
    $v^{n_k}\rightarrow v$ in $C([0,T];H)$.
\end{lemma}
\begin{proof}
Let 
\begin{align*}
    & \zeta^{n_k}=\operatorname{curl}u^{n_k}_t,\quad \zeta_t=\operatorname{curl}u_t,\\
    & \phi^{n_k}_t=\operatorname{curl}F^{n_k}_t,\quad \phi_t=\operatorname{curl}F_t,\\
 & h^{n_k}_t=\operatorname{curl}v^{n_k}_t=\zeta^{n_k}_t-\phi^{n_k}_t,\quad h_t=\operatorname{curl}v_t=\zeta_t-\phi_t.   
\end{align*} $h^{n_k}_t$ (resp. $h_t$) is a weak solution of the vorticity equation
\begin{align}\label{equation vorticty}
    \begin{cases}
        \partial_t h^{n_k} + u^{n_k}\cdot \nabla (h^{n_k} + \phi^{n_k})=0\\
        h^{n_k}_0=\operatorname{curl}x_0^{n_k} 
    \end{cases}
     \quad \left( \text{ resp. }
     \begin{cases}
        \partial_t h + u\cdot \nabla (h + \phi)=0\\
        h_0=\operatorname{curl}x_0
    \end{cases}\right).
 \end{align}
Thanks to \eqref{convergence integral forcing cond1}, \eqref{uniform bound integral forcing cond1}, \eqref{uniform energy W24 euler} $h^{n_k},\ h, \  \phi^{n_k},\ \phi$ satisfy
\begin{align}
    & \phi^{n_k}\rightarrow \phi \quad in \  C([0,T];H^{3-\theta}) \label{convergence forcing vorticity},\\
    & \operatorname{sup}_{k\geq 1} \lVert \phi^{n_k}\rVert_{C([0,T];H^3)}+\lVert \phi\rVert_{C([0,T];H^3)}\leq C(N) \label{uniform forcing vorticity},\\
    & \operatorname{sup}_{k\geq 1}\lVert h^{n_k}\rVert_{C([0,T];W^{1,4})}+\lVert h\rVert_{C([0,T];W^{1,4})}\leq C(\lVert x_0\rVert_{W^{2,4}},N).\label{uniform vorticity}
\end{align}
We need to introduce the stream function $\psi^{n_k}_t$ (resp. $\psi_t$) which is the weak solution of \begin{align}\label{elliptic equation stream function}
    \begin{cases}
        -\Delta \psi^{n_k}_t=h^{n_k}_t\\
        \psi^{n_k}_t|_{\partial D}=0
    \end{cases}
    \quad \left( \text{ resp. }
     \begin{cases}
         -\Delta \psi_t=h_t\\
        \psi_t|_{\partial D}=0
    \end{cases}\right).
\end{align}
By standard elliptic regularity theory, see for example \cite{ambrosio2019lectures}, and the uniform bound \eqref{uniform vorticity}, it holds
\begin{align}
 & \operatorname{sup}_{k\geq 1}\lVert \psi^{n_k}\rVert_{C([0,T];W^{3,4})}+\lVert \psi\rVert_{C([0,T];W^{3,4})}\leq C(\lVert x_0\rVert_{ W^{2,4}},N).\label{uniform stream}
\end{align}
Lastly, introducing 
\begin{align*}
    \alpha^{n_k}_t=\psi^{n_k}_t-\psi_t,\quad g^{n_k}_t=\phi^{n_k}_t-\phi_t,\quad G^{n_k}_t=F^{n_k}_t-F_t,
\end{align*}
It is well-known that $v^{n_k}=-\nabla^{\perp}\psi^{n_k},\ v=-\nabla^{\perp}\psi,$ see for example \cite{marchioro2012mathematical}. With these notations in mind, thanks to equations \eqref{equation vorticty} and \eqref{elliptic equation stream function}, $-\Delta \alpha^{n_k}$ solves in a weak sense 
\begin{align}\label{equation difference solution}
    \partial_t (-\Delta \alpha^{n_k}) + \Big[(-\nabla^\perp \psi^{n_k} + F^{n_k})\cdot \nabla\Big]\Big(-\Delta \alpha^{n_k} + g^{n_k} \Big) 
    +\Big[-\nabla^\perp \alpha^{n_k} + G^{n_k} \cdot \nabla \Big](-\Delta \psi + \phi)=0.
\end{align}
Therefore, arguing as in \cite[Theorem 3.1]{YUDOVICH19631407}, we use $\alpha^{n_k}$ itself as a test function in \eqref{equation difference solution}, obtaining 
\begin{align}\label{energy estimate differences euler 1}
\frac{1}{2}\lVert \nabla \alpha^{n_k}_t\rVert^2 &= \frac{1}{2}\lVert \nabla \alpha^{n_k}_0\rVert^2 + \int_0^t\int_D \Big((-\nabla^\perp \psi_s^{n_k} \cdot\nabla \alpha_s^{n_k})\big(-\Delta \alpha_s^{n_k} + g_s^{n_k}\big) \notag\\ & +F^{n_k}_s\cdot \nabla \alpha^{n_k}_s (-\Delta \alpha^{n_k}_s + g^{n_k}_s) + \big(-\Delta \psi_s + \phi_s\big)\big(G^{n_k}_s \cdot \nabla \alpha^{n_k}_s\big)\Big) dx ds\notag \\ & =\frac{1}{2}\lVert \nabla \alpha^{n_k}_0\rVert^2+I_1(t)+I_2(t)+I_3(t)+I_4(t)+I_5(t),
\end{align}
where \begin{align*}
    I_1(t)&=\int_0^t\int_D \nabla^\perp \psi_s^{n_k} \cdot\nabla \alpha_s^{n_k}\Delta \alpha_s^{n_k} dx ds,\\ 
    I_2(t)&=-\int_0^t\int_D \nabla^\perp \psi_s^{n_k} \cdot\nabla \alpha_s^{n_k}g_s^{n_k} dx ds,\\
    I_3(t)&=-\int_0^t\int_D F_s^{n_k} \cdot\nabla \alpha_s^{n_k}\Delta \alpha_s^{n_k} dx ds,\\
    I_4(t)&=\int_0^t\int_D F_s^{n_k} \cdot\nabla \alpha_s^{n_k}g_s^{n_k} dx ds,\\
    I_5(t)&=\int_0^t \int_D (-\Delta \psi_s + \phi_s\big)\big(G^{n_k}_s \cdot \nabla \alpha^{n_k}_s\big) dx ds
\end{align*}
Therefore we need to understand the behavior of $I_1(t)$, $I_2(t),$ $I_3(t)$, $I_4(t)$, $I_5(t)$ in equation \eqref{energy estimate differences euler 1}.
$I_1(t)$ can be estimate easily integrating by parts, thanks to the uniform bound \eqref{uniform energy W24 euler} and H\"older's inequality. Indeed it holds:
\begin{align*}
    -\int_D v^{n_k}_i \partial_i \alpha^{n_k} \partial_{j,j}\alpha^{n_k} dx&=-\int_D \partial_j v^{n_k}_i \partial_i\alpha^{n_k} \partial_j \alpha^{n_k} dx -\int_D v^{n_k}_i \partial_{i,j}\alpha^{n_k} \partial_j \alpha^{n_k} dx\\ & \leq \lVert v^{n_k}\rVert_{W^{1,\infty}}\lVert \nabla \alpha^{n_k}\rVert^2-\frac{1}{2}\int_D v^{n_k}_i \partial_i \lvert \partial_j \alpha^{n_k}\rvert^2 dx\\ & \leq C(\lVert x_0\rVert_{W^{2,4}},N)\lVert\nabla \alpha^{n_k}\rVert^2.
\end{align*}
Therefore \begin{align}\label{estimate I_1 euler}
    I_1(t)\leq  C(\lVert x_0\rVert_{W^{2,4}},N)\int_0^t \lVert \nabla\alpha^{n_k}_s\rVert^2 ds.
\end{align}
$I_3$ can be estimated similarly integrating by parts, thanks to the uniform bound \eqref{uniform bound integral forcing cond1} and H\"older's inequality:
\begin{align*}
    \int_D F^{n_k}_i\partial_i \alpha^{n_k}\partial_{j,j}\alpha^{n_k}dx&=-\int_D \partial_j F^{n_k}_i\partial_i \alpha^{n_k}\partial_{j}\alpha^{n_k}dx-\int_D  F^{n_k}_i\partial_{i,j} \alpha^{n_k}\partial_{j}\alpha^{n_k}dx\\ & \leq \lVert F^{n_k}\rVert_{W^{1,\infty}}\lVert \nabla\alpha^{n_k}\rVert^2-\frac{1}{2}\int_D F^{n_k}_i \partial_i \lvert \partial_j\alpha^{n_k}\rvert^2 dx\\ & \leq C(N)\lVert \nabla\alpha^{n_k}\rVert^2.
\end{align*}
Therefore \begin{align}\label{estimate I_3 euler}
    I_3(t)\leq C(N)\int_0^t\lVert \nabla \alpha^{n_k}_s\rVert^2 ds.
\end{align}
$I_2(t)$, $I_4(t)$ and $I_5(t)$ can be bounded easily by H\"older and Young inequalities and the uniform bounds \eqref{uniform bound integral forcing cond1}, \eqref{uniform energy W24 euler}, \eqref{uniform forcing vorticity}, \eqref{uniform vorticity}, \eqref{uniform stream}. Indeed it holds:
\begin{align}\label{estimate I_245 euler}
    I_2(t)+I_4(t)+I_5(t) &\leq \int_0^t \left(\lVert v_s^{n_k}\rVert_{L^{\infty}}+\lVert F^{n_k}_s\rVert_{L^{\infty}}\right)\lVert \nabla \alpha^{n_k}_s\rVert\lVert g^{n_k}_s\rVert +\left(\lVert \phi_s\rVert_{L^{\infty}}+\lVert h_s\rVert_{L^{\infty}}\right)\lVert \nabla \alpha^{n_k}_s\rVert\lVert G^{n_k}_s\rVert\notag\\ & \leq \int_0^t \lVert \nabla \alpha^{n_k}_s\rVert^2 ds+C\left(\lVert v^{n_k}\rVert_{C([0,T];W^{1,4})}+\lVert F^{n_k}\rVert_{C([0,T];W^{1,4})}\right)^2\lVert g^{n_k}\rVert_{C([0,T];L^2)}^2\notag\\ & +
    C\left(\lVert \phi\rVert_{C([0,T];W^{1,4})}+\lVert h\rVert_{C([0,T];W^{1,4})}\right)^2\lVert G^{n_k}\rVert_{C([0,T];L^2)}^2\notag\\ & \leq \int_0^t \lVert \nabla \alpha^{n_k}_s\rVert^2 ds+C(\lVert x_0\rVert_{W^{2,4}},N)\left(\lVert g^{n_k}\rVert_{C([0,T];L^2)}^2+\lVert G^{n_k}\rVert_{C([0,T];L^2)}^2\right).
\end{align}
Combining relations \eqref{estimate I_1 euler}, \eqref{estimate I_3 euler} and \eqref{estimate I_245 euler} we get
\begin{align*}
    \frac{1}{2}\lVert \nabla\alpha^{n_k}_t\rVert^2\leq \frac{1}{2}\lVert \nabla\alpha^{n_k}_0\rVert^2+C(\lVert x_0\rVert_{W^{2,4}},N)\int_0^t  \lVert \nabla\alpha^{n_k}_s\rVert^2 ds+C(\lVert x_0\rVert_{W^{2,4}},N)\left(\lVert g^{n_k}\rVert_{C([0,T];L^2)}^2+\lVert G^{n_k}\rVert_{C([0,T];L^2)}^2\right),
\end{align*}
which implies, by Gr\"onwall's Lemma,
\begin{align}\label{final convergence euler}
    \lVert v^{n_k}-v\rVert_{C([0,T];H)}\leq C(\lVert x_0\rVert_{W^{2,4}},N)\left(\lVert x^{n_k}_0-x_0\rVert+\lVert g^{n_k}\rVert_{C([0,T];L^2)}+\lVert G^{n_k}\rVert_{C([0,T];L^2)} \right).
\end{align}
Thanks to our assumptions the thesis follows.
\end{proof}

\subsubsection{Condition 2}\label{subsubsec:cond2 Kato}
Fix $N>0$, let $\tilde{f}^{\eps},\Tilde{f}\in \mathcal{P}_2^N,\ u_0^{\eps}, u_0\in\mathcal{E}_0^{NS}$ such that $\Tilde{f}^{\eps}\rightarrow_{\mathcal{L}}\Tilde{f}$ weakly in $L^2(0,T;H_0),\ u^{\eps}_0\rightarrow u_0$ in $\mathcal{E}_0^{NS}$. We will show that for each sequence $\eps_n\rightarrow 0$, $\mathcal{G}^{\eps_n,NS}\big(u_0^{\eps_n}, {\eps_n}^{1/2} W + \int_0^{\cdot} \tilde f^{\eps_n}_s ds\big)$ converges in law to $\mathcal{G}^{0,NS}(u_0,\int_0^\cdot \tilde f_s ds)$ in the topology of $\E$. This implies the validity of the second condition in \autoref{conditions weak convergence approach}. In order to simplify the notation, we will consider $\eps>0$ in the following dropping the subscript $\eps_n$, having in mind it is a countable family. Since $S^N$ is a Polish space, by Skorokhod's representation theorem we can introduce a further filtered probability space $(\Tilde{\Omega},\Tilde{\mathcal{F}},\Tilde{\mathcal{F}}_t,\Tilde{\mathbb{P}})$ and random variables $f^{\eps}, W^{\eps},f$ such that $(f^{\eps},W^{\eps})$ has the same joint law of $(\Tilde{f}^{\eps},W)$, $f$ has the same law of $\Tilde{f}$ and $f^{\eps}\rightarrow_{\Tilde{\mathbb{P}}-a.s.} f$ in $L^2(0,T;H_0)$, see for example \cite{flandoli2023stochastic} for details. Thanks to \autoref{thm well-posed ns forcing} for each $\epsilon$ we can define $u^{\epsilon}$ as the unique solution of \eqref{NS with forcing} with forcing term $f^{\eps}$, initial condition $u_0^{\eps}$ and Brownian forcing term $W^{\eps}$. The family $\{u^\eps\}_{\eps>0}$ satisfy \autoref{strong kato hp}. Moreover, by \autoref{thm well posed euler forcing} we can define ${u}^E$ as the unique regular solution of \eqref{euler forcing} with forcing term $f$ and initial condition $u_0.$ We will show that $u_0^{\eps}$ converges to ${u}^E$ in probability in $C([0,T];H)$. This implies the validity of Condition 2.\\
Before starting with the computation we recall some facts. In the following, with some abuse of notation, we will simply use $\mathbb{P},\ \mathbb{E}$ instead of $\Tilde{\mathbb{P}},\ \Tilde{\mathbb{E}}.$ Fix $\theta>0$ arbitrarily small and define $F^{\eps}_t=\int_0^t f^{\eps}_s ds,\ F_t=\int_0^t f_s ds$. By hypothesis \autoref{hypothesis noise}, $H_0\hookrightarrow D((-A)^2)$. This implies, see for example \cite[Proposition 26]{flandoli2023stochastic} that \begin{align}\label{convergence integral forcing NS}
    F^{\eps}\rightarrow_{{\mathbb{P}}-a.s.} F \quad\textit{in }C([0,T];D((-A)^{2-\theta})).
\end{align} Obviously
\begin{align}\label{uniform bound integral forcing NS}
\operatorname{sup}_{\eps>0}\lVert F^{\eps}\rVert_{C([0,T];D((-A)^2))}+\lVert F\rVert_{C([0,T];D((-A)^2))}\leq C(N)\quad {\mathbb{P}}-a.s.  
\end{align}
Starting from \eqref{ito NS}, Burkholder-Davis-Gundy inequality, Gr\"onwall's lemma and the convergence of $u_0^{\eps}$ to $u_0$ imply
\begin{align}\label{energy estimate NS}
    \operatorname{sup}_{\eps>0}\left\{\mathbb{E}\left[\operatorname{sup}_{t\in [0,T]}\lVert u^{\eps}_t\rVert_H^2\right]+\eps\mathbb{E}\left[\int_0^T \lVert \nabla u^{\eps}_s\rVert^2 ds  \right]\right\}\leq C(N,\lVert u_0\rVert).
\end{align}
In order to show the convergence of $u^{\eps}$ to $u^E$ we will introduce $z^{\eps}=\int_0^t e^{\eps A (t-s)}f^{\eps}_sds,\ v^{\eps}=u^{\eps}-z^{\eps},\ v^E=u^E-F$ and show separately the convergence of $z^{\eps}$ to $F$ and of $v^\eps$ to $v^E.$ The Strong Kato Condition \autoref{strong kato hp} will play a role only in the convergence of $v^\eps$ to $v^E.$ 
We start with the convergence of $z^{\eps}$ to $F$.
\begin{lemma}\label{convergence forcing}
    For each $\theta>0$, $z^{\eps}_t\rightarrow F$ in $C([0,T];D((-A)^{\gamma-1/2-\theta}))\quad \mathbb{P}-a.s.$ and in $L^2(\Omega,\mathbb{P})$.
\end{lemma}
\begin{proof}
$z^{\eps}$ can be rewritten as 
\begin{align*}
    z^{\eps}_t=\int_0^t f^{\eps}_sds+\eps\int_0^t Ae^{\eps(t-s)A}f^{\eps}_s ds&=I_1+I_2.
\end{align*}
$I_1\rightarrow F\in C([0,T];D((-A)^{\gamma-1/2-\theta})$ $\mathbb{P}-a.s.$ thanks to \eqref{convergence integral forcing NS}. Moreover, since \eqref{uniform bound integral forcing NS} holds, previous convergence holds also in $L^2(\Omega,\mathbb{P})$ by Lebesgue theorem. It remains to show that $I_2\rightarrow 0 $ properly. The $\mathbb{P}-a.s.$ convergence can be obtained as follows
\begin{align*}
    \eps \operatorname{sup}_{t\in [0,T]}\left\lVert \int_0^t A e^{\eps(t-s)A}f^{\eps}_sds \right\rVert_{D((-A)^{\gamma-1/2-\theta}}&\leq  \eps  \operatorname{sup}_{t\in [0,T]}\int_0^t \left\lVert A^{1/2+\gamma-\theta} e^{\eps(t-s)A}f^{\eps}_sds \right\rVert\\ & \leq
    \eps  \operatorname{sup}_{t\in [0,T]}\int_0^t \left\lVert A^{1/2-\theta} e^{\eps(t-s)A}A^{\gamma} f^{\eps}_s \right\rVert ds\\ & \leq 
    \eps^{1/2+\theta}\operatorname{sup}_{t\in [0,T]}\int_0^t \frac{1}{(t-s)^{1/2-\theta}}\lVert f^{\eps}_s\rVert_{D(-A)^{\gamma}}ds\\ & \leq C(N)\eps^{1/2+\theta}\rightarrow 0\quad \mathbb{P}-a.s.
\end{align*}
Since previous bound is uniform in $\omega\in \Omega$ the convergence holds also in  $L^2(\Omega,\mathbb{P})$ and the thesis follows.
\end{proof}
In order to prove the convergence of $v^{\eps}$ to $v^E$ we observe that they solve in a sense analogous to \autoref{weak solution NS forcing}, \autoref{weak euler}
 \begin{align}\label{auxiliary eps}
       dv^{\eps}+ P\left(\left(v^{\eps}+z^{\eps}\right)\cdot\nabla\left(v^{\eps}+z^{\eps}\right)\right)dt + \eps Av^{\eps} dt  = \sqrt{\eps}dW_t,
\end{align}
\begin{align}\label{auxiliary 0}
     \partial_t v^{E} + P(u^E\cdot\nabla u^E)   =0.
\end{align} 
By triangle inequality and the uniform bound guaranteed by \autoref{convergence forcing}, 
 estimates analogous to \eqref{energy estimate NS}, \eqref{Energy Euler } hold for $v^{\eps}$ and $v^E$, too.
We observe that, thanks to the regularity of $u^E$ guaranteed by \autoref{thm well posed euler forcing}
\begin{align}\label{energy derivative time vE}
    \lVert \partial_t v^{E}\rVert_{C([0,T];L^{\infty}(D)))}& \lesssim  \lVert \partial_t v^{E}\rVert_{C([0,T];W^{1,4}(D)) }\lesssim \lVert P(u^E\cdot\nabla u^E)\rVert_{C([0,T];W^{1,4}(D)) }\notag\\ & \lesssim \lVert u^E\rVert_{C([0,T];W^{1,4}(D))}^2\leq  C(N,u_0)\quad \mathbb{P}-a.s.
\end{align}
Following the idea of \cite{kato1984remarks}, let $v$ the corrector of the boundary layer of width $\delta=\delta(\eps)$, i.e. a divergence free vector field with support in a strip of the boundary of width $\delta$ such that $v^E-v\in V$ and $\mathbb{P}-a.s.$ uniformly in $t\in [0,T],\ \omega\in \Omega$ \begin{align}\label{property bl corrector}
   & \lVert v_t\rVert_{L^{\infty}(D)}\leq C(N,u_0), \ \lVert v_t\rVert\leq C(N,u_0)\delta^{\frac{1}{2}}, \ \lVert\partial_t v_t\rVert\leq C(N,u_0)\delta^{\frac{1}{2}},\notag \\ & \lVert\nabla v_t\rVert_{L^{\infty}(D)}\leq C(N,u_0)\delta^{-1}, 
\  \lVert\nabla v_t\rVert\leq C(N,u_0)\delta^{-1/2},  \  \lVert\rho\nabla v_t\rVert_{L^{\infty}(D)}\leq C(N,u_0),\notag\\ & \lVert\rho^2 \nabla v_t\rVert_{L^{\infty}(D)}\leq C(N,u_0)\delta, \  \lVert\rho \nabla v_t\rVert\leq C(N,u_0)\delta^{\frac{1}{2}},
\end{align}
$\rho$ being the distance function to $\partial D.$
Now we are ready to show the convergence of $v^{\eps}$ to $v^E$.
\begin{lemma}\label{convergence v^eps v^E}
    $v^{\eps}\rightarrow v^{E}$ in $C([0,T];H)$ in probability.
\end{lemma}
\begin{proof}
 Arguing as in \cite[Theorem 9]{luongo2022inviscid} one can show that the following relations hold true.
 \begin{align}\label{ito veps}
     \lVert v^{\eps}_t\rVert^2+2\eps\int_0^t \lVert \nabla v^{\eps}\rVert^2 ds &=\lVert u_0^{\eps}\rVert^2-2\int_0^t b(v^{\eps}+z^{\eps},z^{\eps},v^{\eps})ds  +2\eps^{1/2}\int_0^t\langle dW^{\eps}_s,v^{\eps}\rangle+\eps t Tr(Q) \quad \mathbb{P}-a.s.,
 \end{align}   
 \begin{align}\label{ito vE}
     \lVert v^{E}_t\rVert^2&=\lVert u_0\rVert^2-2\int_0^t b(v^E+F,F,v^{E})ds\quad \mathbb{P}-a.s. 
 \end{align}   
Exploiting relations \eqref{ito veps}, \eqref{ito vE} we can study $\lVert v^{\eps}_t-v^E_t\rVert^2$. Indeed, it holds
\begin{align}\label{step 1 main proof}
    \lVert v^{\eps}_t-v^E_t\rVert^2&=\lVert v^{\eps}_t\rVert^2+\lVert v^E_t\rVert^2-2\langle v^{\eps}_t,v^E_t\rangle \notag\\ & \leq \lVert u_0^{\eps}\rVert^2-2\int_0^t b(v^{\eps}_s+z^{\eps}_s,z^{\eps}_s,v^{\eps}_s)ds  +2\eps^{1/2}\int_0^t\langle dW^{\eps}_s,v^{\eps}_s\rangle+\eps t Tr(Q)\notag\\ & + \lVert u_0\rVert^2-2\int_0^t b(v^E_s+F_s,F_s,v^{E}_s)ds -2\langle v^{\eps}_t,v^E_t-v_t\rangle-2\langle v^{\eps}_t,v_t\rangle.
\end{align}
Thanks to the fact that $v^E-v\in C^1([0,T];V)$ we can rewrite $\langle v^{\eps}_t,v^E_t-v_t\rangle$ via It\^o formula:  $\langle v^{\eps}_t,v^E_t-v_t\rangle=\langle u^{\eps}_0,u_0-v_0\rangle+\int_0^t \langle v^{\eps}_s,\partial_s(v^E_s-v_s)\rangle ds +\int_0^t \langle d v^{\eps}_s,v^E_s-v_s\rangle . $ Therefore
\begin{align}\label{step 2 main proof}
    \langle v^{\eps}_t,v^E_t-v_t\rangle&=\langle u^{\eps}_0,u_0-v_0\rangle+\int_0^t \langle v^{\eps}_s,\partial_s(v^E_s-v_s)\rangle ds-\eps\int_0^t \langle \nabla v^{\eps}_s,\nabla (v^E_s-v_s)\rangle ds\notag\\ & +\int_0^t b(v^{\eps}_s+z^{\eps}_s, v^E_s-v_s,v^{\eps}_s+z^{\eps}_s)ds+\eps^{1/2}\langle W^{\eps}_t,v^{E}_t-v_t\rangle-\eps^{1/2}\int_0^t\langle W^{\eps}_s,\partial_s(v^E_s-v_s)\rangle ds.
\end{align}
Let us observe that by assumptions \begin{align*}
     \lVert u_0^{\eps}\rVert^2+\lVert u_0\rVert^2-2\langle u^{\eps}_0,u_0\rangle=\lVert u_0^{\eps}-u_0\rVert^2=o(1),\quad \eps t Tr(Q)\leq \eps T Tr(Q)=o(1).
\end{align*} Moreover, thanks to the properties of the boundary layer corrector \eqref{property bl corrector}, $\mathbb{P}-a.s.$ it holds \begin{align*}
    &\langle u_0^{\eps},v_0\rangle\leq C(u_0,N)\delta^{1/2}=o(1), \\ &\langle v^{\eps}_t,v_t\rangle\leq C(N,u_0)\delta^{1/2}\operatorname{sup}_{t\in [0,T]}\lVert v_t^{\eps}\rVert,\\ &\eps^{1/2}\langle W^{\eps}_t,v^{E}_t-v_t\rangle-{\eps}^{1/2}\int_0^t\langle W^{\eps}_s,\partial_s(v^E_s-v_s)\rangle ds\leq \eps^{1/2}C(N,u_0)\operatorname{sup}_{t\in [0,T]}\lVert W^{\eps}_t\rVert.
\end{align*}  Exploiting these facts, inserting relation \eqref{step 2 main proof} in \eqref{step 1 main proof} we obtain
\begin{align}\label{step 3 main proof}
    \lVert v^{\eps}_t-v^E_t\rVert^2& \leq o(1)+\delta^{1/2}C(N,u_0)\operatorname{sup}_{t\in [0,T]}\lVert v^{\eps}_t\rVert \notag\\ & +\eps^{1/2}C(N,u_0)\operatorname{sup}_{t\in [0,T]}\lVert W^{\eps}_t\rVert-2\int_0^t b(v^{\eps}_s+z^{\eps}_s,z^{\eps}_s,v^{\eps}_s)ds \notag\\ & +2\eps^{1/2}\int_0^t\langle dW^{\eps}_s,v^{\eps}_s\rangle-2\int_0^t b(v^E_s+F_s,F_s,v^{E}_s)ds\notag\\ &-2\int_0^t \langle v^{\eps}_s,\partial_s(v^E_s-v_s)\rangle ds+2\eps\int_0^t \langle \nabla v^{\eps}_s,\nabla(v^E_s-v_s)\rangle ds\notag\\ & -2\int_0^t b(v_s^{\eps}+z_s^{\eps}, v_s^E-v_s,v_s^{\eps}+z_s^{\eps})ds \quad\mathbb{P}-a.s.
\end{align}
In order to understand the behavior of $\int_0^t \langle v^{\eps}_s,\partial_s(v^E_s-v_s)\rangle ds$, we observe that, thanks to \eqref{property bl corrector},\begin{align*}
    \int_0^t \langle v^{\eps}_s,\partial_s v_s\rangle ds\leq \delta^{1/2} C(N,u_0)\operatorname{sup}_{t\in [0,T]}\lVert v_t^{\eps}\rVert \quad\mathbb{P}-a.s.
\end{align*} Moreover, since $v^{E}$ satisfies \eqref{auxiliary 0}, we have $\int_0^t \langle v_s^{\eps},\partial_s v^E_s\rangle ds=-\int_0^t b(v^{E}_s+F_s,v^E_s+F_s,v^{\eps}_s) ds $. Let us rewrite the trilinear forms appearing  \eqref{step 3 main proof}:
\begin{align}\label{step 4 main proof}
    & b(v^{E}+F,v^E+F,v^{\eps})-b(v^{\eps}+z^{\eps},z^{\eps},v^{\eps})-b(v^E+F,F,v^E)\notag\\ &-b(v^{\eps}+z^{\eps},v^E-v,v^{\eps}+z^{\eps})\notag \\ & = b(v^{E},v^E,v^{\eps}-v^E)+b(v^{E}+F,F,v^{\eps})+b(F,v^E+F,v^{\eps})\notag\\ & -b(v^{\eps}+z^{\eps},z^{\eps},v^{\eps})-b(v^E+F,F,v^E)-b(v^{\eps},v^E,v^{\eps}-v^E)\notag\\ &-b(z^{\eps},v^E-v, v^{\eps}+z^{\eps})-b(v^{\eps},v^E,z^{\eps})+b(v^{\eps},v,v^{\eps}+z^{\eps}).
\end{align}
By simple computations the terms in \eqref{step 4 main proof} can be rewritten as:
\begin{align}\label{step 5 main proof}
   \lvert  b(v^E,v^E,v^{\eps}-v^E)-b(v^{\eps},v^E,v^{\eps}-v^E)\rvert\leq \lVert \nabla v^E\rVert_{L^{\infty}}\lVert v^{\eps}-v^E\rVert^2.
\end{align}
\begin{align}\label{step 6 main proof}
b(v^{\eps},v,v^{\eps}+z^{\eps})-b(z^{\eps},v^E-v, v^{\eps}+z^{\eps})=b(u^{\eps},v,u^{\eps})-b(z^{\eps},v^{E},u^{\eps}). \end{align}
\begin{align}\label{step 7 main proof}
    -b(v^{\eps}+z^{\eps},z^{\eps},v^{\eps})+b(v^{\eps},v^E,z^{\eps})=b(v^{\eps},z^{\eps},v^E-v^{\eps})-b(z^{\eps},z^{\eps},v^{\eps}).
\end{align}
\begin{align}\label{step 8 main proof}
    & b(F,v^E+F,v^{\eps})+b(v^{E}+F,F,v^{\eps})-b(v^E+F,F,V^E)\notag\\ &=b(v^E,F,v^{\eps}-v^E)+b(F,F,v^{\eps}-v^E)+b(F,v^E,v^{\eps}).
\end{align}
Preliminarily, let us rewrite the last terms in each of \eqref{step 6 main proof}, \eqref{step 7 main proof} and \eqref{step 8 main proof} obtaining
\begin{align}\label{step 9 main proof}
    & -b(z^{\eps},v^{E},u^{\eps})-b(z^{\eps},z^{\eps},v^{\eps})+b(F,v^E,v^{\eps})= b(F-z^{\eps},v^E,v^{\eps})+b(z^{\eps},v^{\eps}-v^E,z^{\eps}).
\end{align}
Let us leave out $b(u^{\eps},v,u^{\eps})$ from our analysis for a moment. Indeed, it well be treated differently. Then considering the other terms appearing in \eqref{step 6 main proof}, \eqref{step 7 main proof} and \eqref{step 8 main proof} and exploiting \eqref{step 9 main proof}, we have
\begin{align}\label{step 10 main proof}
    & b(v^{\eps},z^{\eps},v^E-v^{\eps})+b(v^E,F,v^{\eps}-v^E)+b(F,F,v^{\eps}-v^E) \notag\\ & +b(F-z^{\eps},v^E,v^{\eps})-b(z^{\eps},z^{\eps},v^{\eps}-v^E)\pm b(z^{\eps},F,v^{\eps}-v^E)\notag\\ & \pm b(v^E,z^{\eps},v^E-v^{\eps})\notag\\ &= b(v^E+z^{\eps}, F-z^{\eps}, v^{\eps}-v^E)+b(v^{\eps}-v^E,z^{\eps},v^E-v^{\eps})\notag \\ &+b(F-z^{\eps},v^E,v^{\eps}) +b(F-z^{\eps},F,v^{\eps}-v^E).
\end{align}
Therefore we can simplify \eqref{step 3 main proof} and it holds
\begin{align}\label{step 11 main proof}
\lVert v^{\eps}_t-v^E_t\rVert^2& \leq o(1)+\delta^{1/2}C(N,u_0)\operatorname{sup}_{t\in [0,T]}\lVert v^{\eps}_t\rVert+\eps^{1/2}C(N,u_0)\operatorname{sup}_{t\in [0,T]}\lVert W^{\eps}_t\rVert \notag\\ & +2\eps^{1/2}\int_0^t\langle dW^{\eps}_s,v^{\eps}_s\rangle+2\eps\int_0^t \langle \nabla v^{\eps}_s,\nabla(v^E_s-v_s)\rangle ds \notag\\ & +2\lVert \nabla v^E\rVert_{L^{\infty}(0,T;L^{\infty}}\int_0^t \lVert v^{\eps}_s-v^E_s\rVert^2 ds+2\int_0^t b(u^{\eps}_s,v_s,u^{\eps}_s) ds \notag\\ & +2\int_0^t \lVert v^{\eps}_s-v^E_s\rVert \lVert \nabla F_s\rVert_{L^{4}}\lVert F_s-z^{\eps}_s\rVert_{L^4} ds\notag\\ & +2\int_0^t \lVert v^{\eps}_s-v^E_s\rVert \lVert \nabla(F_s-z^{\eps}_s)\rVert \lVert v^E_s+z^{\eps}_s\rVert_{L^{\infty}}ds  \notag \\ & +2\lVert \nabla z^{\eps}\rVert_{L^{\infty}(0,T;L^{\infty})}\int_0^t \lVert v^{\eps}_s-v^E_s\rVert^2 ds+2\int_0^t \lVert \nabla v^{E}_s\rVert\lVert v^{\eps}_s\rVert \lVert F_s-z^{\eps}_s\rVert_{L^{\infty}} ds.
\end{align}
Now we can treat the term $\eps\int_0^t \langle \nabla v^{\eps}_s,\nabla(v^E_s-v_s)\rangle ds$ exploiting the properties of the boundary layer corrector \eqref{property bl corrector} and the convergence $z^{\eps}\rightarrow F$ in $C([0,T];D((-A)^{\gamma-\frac{1}{2}-\theta}))$ $\mathbb{P}$-a.s. 
\begin{align}\label{step 12 main proof}
    & 2\eps\int_0^t \langle \nabla v^{\eps}_s,\nabla(v^E_s-v_s)\rangle ds \notag \\& =2\eps\int_0^t \langle \nabla u^{\eps},\nabla (v^E_s-v_s)\rangle ds-2\eps \int_0^t \langle \nabla z^{\eps}_s,\nabla(v^E_s-v_s)\rangle ds\notag\\ &\leq 2\eps \int_0^t \lVert \nabla u^{\eps}_s\rVert \lVert \nabla v^E_s\rVert ds +2\eps \int_0^t \lVert \nabla u^{\eps}_s\rVert_{L^2(\Gamma_\delta)}\lVert \nabla v_s\rVert ds\notag\\ &+2\eps\int_0^t \lVert (-A)^{1/2}z^{\eps}_s\rVert \lVert \nabla(v^E_s-v_s)\rVert ds  \notag\\ & \leq 2\eps \int_0^t \lVert \nabla u^{\eps}_s\rVert \lVert \nabla v^E_s\rVert ds +\delta^{-1/2}\eps C(N,u_0) \int_0^t \lVert \nabla u^{\eps}_s\rVert_{L^2(\Gamma_\delta)}\notag\\ &+\delta^{-1/2}\eps C(N,u_0)\int_0^t \lVert (-A)^{1/2}z^{\eps}_s\rVert  ds
\end{align}
%\notag\\  & = o(1)+2\eps \int_0^t \lVert \nabla u^{\eps}\rVert \lVert \nabla v^E\rVert ds +2\eps \int_0^t \lVert \nabla u^{\eps}\rVert_{L^2(\Gamma_\delta)}\lVert \nabla v\rVert ds\notag\\ &+2\eps\int_0^t \lVert (-A)^{1/2}z^{\eps}\rVert \lVert (-A)^{1/2}(v^E-v)\rVert ds  \notag\\ & +2\eps \int_0^t \lVert \nabla u^{\eps}\rVert \lVert \nabla v^E\rVert ds +C(N,u_0)\delta^{-1/2}\eps \int_0^t \lVert \nabla u^{\eps}\rVert_{L^2(\Gamma_\delta)}

The term $\int_0^t b(u^{\eps},v,u^{\eps})ds $ is the classical term in the analysis of the inviscid limit in the Kato's regime, it can be estimate by
\begin{align}\label{step 13 main proof}
    \left\lvert\int_0^t b(u^{\eps}_s,v_s,u^{\eps}_s)ds \right\rvert \leq \delta C(N,u_0) \int_0^t\lVert\nabla u^{\eps}_s\rVert^2_{L^2(\Gamma_{\delta})} ds
\end{align}
see \cite[Equation 5.8]{kato1984remarks}.
Combining estimates \eqref{step 11 main proof}, \eqref{step 12 main proof} and \eqref{step 13 main proof}, choosing $\delta=c\eps$, where $c$ is the constant appearing in \autoref{strong kato hp}, it holds
\begin{align}\label{step 14 main proof}
\lVert v^{\eps}_t-v^E_t\rVert^2& \leq o(1)+\eps^{1/2} C(N,u_0)\operatorname{sup}_{t\in [0,T]}\lVert v^{\eps}_t\rVert +2{\eps}^{1/2}\int_0^t\langle dW^{\eps}_s,v^{\eps}_s\rangle+\eps C(N,u_0) \int_0^t \lVert \nabla u^{\eps}_s\rVert ds \notag \\ &+\eps^{1/2}C(N,u_0)\operatorname{sup}_{t\in [0,T]}\lVert W^{\eps}_t\rVert+\eps^{1/2} C(N,u_0) \int_0^t \lVert \nabla u^{\eps}_s\rVert_{L^2(\Gamma_{c\eps})} ds  +\eps^{1/2} C(N,u_0) \int_0^t \lVert \nabla u^{\eps}_s\rVert_{L^2(\Gamma_{c\eps})} \notag\\ & +2\lVert \nabla v^E\rVert_{L^{\infty}(0,T;L^{\infty})}\int_0^t \lVert v^{\eps}-v^E\rVert^2 ds+ \eps C(N,u_0)\int_0^t\lVert\nabla u^{\eps}_s\rVert^2_{L^2(\Gamma_{c\eps})} ds \notag\\ & +2(1+\lVert \nabla z^{\eps}\rVert_{L^{\infty}(0,T;L^{\infty})})\int_0^t \lVert v^{\eps}_s-v^E_s\rVert^2 ds +2\int_0^t \lVert v^{\eps}_s\rVert\lVert \nabla v^{E}_s\rVert \lVert F_s-z^{\eps}_s\rVert_{L^{\infty}} ds.
\end{align}
Therefore, by Gr\"onwall's inequality, equation \eqref{step 14 main proof} implies 
\begin{align}\label{step 15 main proof}
    \operatorname{sup}_{t\in [0,T]}\lVert v^{\eps}_t-v^{E}_t\rVert^2& \leq e^{2T(1+\lVert \nabla v^{E}\rVert_{L^{\infty}(0,T;L^{\infty})}+\lVert \nabla z^{\eps}\rVert_{L^{\infty}(0,T;L^{\infty})})}\times\notag\\ & \bigg(o(1)+\eps^{1/2} C(N,u_0)\operatorname{sup}_{t\in [0,T]}\lVert v^{\eps}_t\rVert +2{\eps}^{1/2}\operatorname{sup}_{t\in [0,T]}\left\lvert\int_0^t\langle dW^{\eps}_s,v^{\eps}_s\rangle\right\rvert\notag\\ &+\eps^{1/2}C(N,u_0)\operatorname{sup}_{t\in [0,T]}\lVert W^{\eps}_t\rVert+\eps C(N,u_0) \int_0^T \lVert \nabla u^{\eps}_s\rVert ds +\eps^{1/2} C(N,u_0) \int_0^T \lVert \nabla u^{\eps}_s\rVert_{L^2(\Gamma_{c\eps})} \notag\\ & + \eps C(N,u_0)\int_0^T\lVert\nabla u^{\eps}_s\rVert^2_{L^2(\Gamma_{c\eps})} ds  +2\int_0^T \lVert v^{\eps}_s\rVert\lVert \nabla v^{E}_s\rVert \lVert F_s-z^{\eps}_s\rVert_{L^{\infty}} ds\bigg)
\end{align}
 Since $\gamma\geq 2$ we can find $\theta>0$ small enough such that $D((-A)^{\gamma-1/2-\theta})\hookrightarrow W^{1,\infty}$. Therefore,  $\lVert \nabla z^{\eps}\rVert_{L^{\infty}(0,T;L^{\infty})}$ is $\mathbb{P}-a.s.$ bounded  by $C(N,u_0)$ from \autoref{convergence forcing}. Similarily, from \autoref{thm well posed euler forcing}, $\lVert \nabla v^{E}\rVert_{L^{\infty}(0,T;L^{\infty})}\leq C(N,u_0)$ $\quad \mathbb{P}-a.s$. \\
Therefore $e^{2T(1+\lVert \nabla v^{E}\rVert_{L^{\infty}_tL^{\infty}_x}+\lVert \nabla z^{\eps}\rVert_{L^{\infty}_tL^{\infty}_x})}\leq C(N,u_0) \quad \mathbb{P}-a.s.$ This means that, in order to show that \autoref{convergence v^eps v^E} holds, it is enough to prove that 
\begin{align}\label{step 16 main proof}
    &\eps^{1/2}C(N,u_0)\left(\operatorname{sup}_{t\in [0,T]}\lVert W^{\eps}_t\rVert+\operatorname{sup}_{t\in [0,T]}\left\lvert\int_0^t\langle dW^{\eps}_s,v^{\eps}_s\rangle\right\rvert\right)+\eps C(N,u_0) \int_0^T \lVert \nabla u^{\eps}_s\rVert ds\notag\\ &+\eps^{1/2} C(N,u_0) \int_0^T \lVert \nabla u^{\eps}_s\rVert_{L^2(\Gamma_{c\eps})} + \eps C(N,u_0)\int_0^T\lVert\nabla u^{\eps}_s\rVert^2_{L^2(\Gamma_{c\eps})} ds  +2\int_0^T \lVert v^{\eps}_s\rVert\lVert \nabla v^{E}_s\rVert \lVert F_s-z^{\eps}_s\rVert_{L^{\infty}} ds\notag\\ & \rightarrow 0 \quad \text{in Probability.}
\end{align}
The terms \begin{align*}
   \eps^{1/2} C(N,u_0)\int_0^T \lVert \nabla u^{\eps}_s\rVert_{L^2(\Gamma_{c\eps})} + \eps C(N,u_0)\int_0^T\lVert\nabla u^{\eps}_s\rVert^2_{L^2(\Gamma_{c\eps})} ds\rightarrow 0  \quad \text{in Probability}
\end{align*}
thanks to \autoref{strong kato hp}.
The terms
\begin{align*}
     \eps^{1/2} C(N,u_0)\left(\sup_{t\in [0,T]}\lVert W^{\eps}_t\rVert\right)+{\eps}^{1/2}\operatorname{sup}_{t\in [0,T]}\left\lvert\int_0^t\langle dW^{\eps}_s,v^{\eps}_s\rangle\right\rvert+\eps C(N,u_0) \int_0^T \lVert \nabla u^{\eps}_s\rVert ds \rightarrow 0 \quad \text{in Probability}
\end{align*}
since it holds by Burkholder-Davis-Gundy inequality, H\"older inequality and \eqref{energy estimate NS}
\begin{align*}
    &\mathbb{E}\left[\sup_{t\in [0,T]}\lVert W^{\eps}_t\rVert\right]+\mathbb{E}\left[\operatorname{sup}_{t\in [0,T]}\left\lvert\int_0^t\langle dW^{\eps}_s,v^{\eps}_s\rangle\right\rvert\right]+\eps^{1/2}  \mathbb{E}\left[\int_0^T \lVert \nabla u^{\eps}_s\rVert ds\right] \notag\\ &\leq C+C\mathbb{E}\left[\operatorname{sup}_{t\in [0,T]}\lVert v^{\eps}_t\rVert^2\right]^{1/2}+C(T)\mathbb{E}\left[\eps\int_0^T\lVert \nabla u^{\eps}_s\rVert^2 ds \right]\\ & \leq C(N,u_0,T).
\end{align*}
Lastly \begin{align*}
\int_0^T \lVert v^{\eps}_s\rVert\lVert \nabla v^{E}_s\rVert \lVert F_s-z^{\eps}_s\rVert_{L^{\infty}} ds\rightarrow 0  \quad \text{in Probability}    
\end{align*}
thanks to \autoref{convergence forcing}, \eqref{energy estimate NS} and \eqref{Energy Euler }. Indeed it holds:
\begin{align*}
    \mathbb{E}\left[\int_0^T \lVert v^{\eps}_s\rVert\lVert \nabla v^{E}_s\rVert \lVert F_s-z^{\eps}_s\rVert_{L^{\infty}} ds\right]& \leq C(N, u_0)\mathbb{E}\left[\int_0^T \lVert v^{\eps}_s\rVert \lVert F_s-z^{\eps}_s\rVert_{L^{\infty}} ds\right]\\ & \leq 
    C(N, u_0)\mathbb{E}\left[\int_0^T \lVert v^{\eps}_s\rVert^2 \right]^{1/2}\mathbb{E}\left[\int_0^T \lVert F_s-z^{\eps}_s\rVert_{L^{\infty}}^2 ds\right]^{1/2}\rightarrow 0.
\end{align*}
Therefore \eqref{step 16 main proof} holds and the thesis follows.
\end{proof}
Combining \autoref{convergence forcing} and \autoref{convergence v^eps v^E} the second condition in \autoref{conditions weak convergence approach} holds.

\begin{proof}[Proof of \autoref{main thm LDP NS}]
Since we already checked the validity of condition 1 and condition 2 in \autoref{conditions weak convergence approach}, it remains to show that for each $v\in \mathcal{E}$ the map $u_0\rightarrow I^{NS}_{u_0}(v)$ is a lower continuous map from $\mathcal{E}_0^{NS}$ to $[0,+\infty]$ in order to apply \autoref{weak convergence approach abstract thm} and complete the proof. The arguments goes in this way. Fix $u_0\in \mathcal{E}_0^{NS}$ and a family $\{u_0^n\}_{n\in \mathbb{N}}\subseteq \mathcal{E}_0^{NS} $ converging to $u_0$. Without loss of generality we may assume $\operatorname{liminf}_{n\rightarrow +\infty}I^{NS}_{u_0^n}(v)=M<+\infty$ otherwise we have nothing to prove. Therefore, thanks to the well-posedness of the Euler equations guaranteed by \autoref{thm well posed euler forcing},  there exists a subsequence $n_k$ and family $\{f^{n_k}\}_{n_k\in\mathbb{N}}\subseteq S^{2M}$ such that $\mathcal{G}^{NS,0}(u_0^{n_k},\int_0^\cdot f^{n_k}_s ds)=v$. Moreover $ f^{n_k}\in S^{2M} $ for all $k$.
Up to passing to a further subsequence, which we continue to denote by $f^{n_k}$ for simplicity of notation, there exists $f\in S^{2M}$ such that $f^{n_k}\rightharpoonup f$ in $L^2(0,T;H_0)$. Thanks to \eqref{convergence integral forcing cond1} and \eqref{final convergence euler} it follows that $\mathcal{G}_0^{NS}: \mathcal{E}_0^{NS}\times L^2([0,T];H_0)\rightarrow \mathcal{E}$ is a continuous map endowing $L^2([0,T];H_0)$ with the weak topology. Therefore $\mathcal{G}^{NS,0}(u_0,f)=v$ and from the lower semicontinuity of the norm with respect to the weak convergence the thesis follow immediately: 
\begin{align*}
    I_{u_0}(v)\leq \frac{1}{2}\int_0^T \lVert f_s\rVert^2_{H_0} ds \leq \operatorname{liminf}_{k\rightarrow +\infty}\frac{1}{2}\int_0^T \lVert f^{n_k}_s\rVert^2_{H_0} ds\leq M=\operatorname{liminf}_{n\rightarrow +\infty} I_{u^n_0}(v).
\end{align*}
%From  \eqref{convergence integral forcing cond1} and \eqref{final convergence euler} it follows that $\mathcal{G}_0^{NS}: \mathcal{E}_0^{NS}\times L^2([0,T];H_0)\rightarrow \mathcal{E}$ is a continuous map endowing $L^2([0,T];H_0)$ with the weak topology. Therefore for each $f\in \mathcal{E},\ N\geq 0 ,\  (\mathcal{G}_0^{NS})^{-1}(f)$ is a closed set in  $\mathcal{E}_0^{NS}\times L^2([0,T];H_0)$ and the same holds for $(\mathcal{G}_0^{NS})^{-1}(f)\cap \mathcal{E}_0\times S^N$. Denoting by $\pi^N_1:\mathcal{E}_0^{NS}\times S^N\rightarrow \mathcal{E}_0^{NS}$ the projection on $\mathcal{E}_0^{NS}$, which is closed since $S^N$ is compact, then  
%\begin{align*}
%    \{I^{NS}_{u_0}(v)\leq N\}=\pi_1^N((\mathcal{G}_0^{NS})^{-1}(f)\cap \mathcal{E}_0\times S^N)
%\end{align*}
%is closed and the thesis follows.%
\end{proof}
\begin{remark}\label{remark inverse kato}
As it is classical in the analysis of the inviscid limit in bounded domains, \autoref{strong kato hp} for the forcing terms $f^{\eps}$, $f$ is implied by the convergence in probability of $u^{\eps}$ to $u^E$ in the probability space introduced by Skorokhod's representation theorem. Let us consider \eqref{ito NS} for $t=T$ and take the limsup of this expression for $\eps \rightarrow 0$. It holds \begin{align}\label{proof inverse convergence}
        2\operatorname{limsup}_{\eps\rightarrow 0}\eps \int_0^T \lVert \nabla u^{\eps}_s\rVert^2_{L^2}ds &\leq \operatorname{\limsup}_{\eps\rightarrow 0}\lVert u_0^{\eps}\rVert^2+\operatorname{\limsup}_{\eps\rightarrow 0}\eps T\sum_{k\in K}\lVert \sigma_k\rVert^2\\ & +2\operatorname{\limsup}_{\eps\rightarrow 0}\sqrt{\eps}\int_0^T \langle u^{\eps}_s,dW^{\eps}_s\rangle+2\operatorname{\limsup}_{\eps\rightarrow 0} \int_0^T\langle f^{\eps}_s, u_s^{\eps}\rangle ds\notag \\ & +\limsup_{\eps \rightarrow 0}\{-\lVert u^{\eps}_T\rVert^2\}.
\end{align}
Under our assumptions it follows immediately that 
\begin{align*}
    \operatorname{\limsup}_{\eps\rightarrow 0}\lVert u_0^{\eps}\rVert^2+\operatorname{\limsup}_{\eps\rightarrow 0}\eps T\sum_{k\in K}\lVert \sigma_k\rVert^2+ \limsup_{\eps \rightarrow 0}\{-\lVert u^{\eps}_T\rVert^2\}=\lVert u_0\rVert^2-\lVert u^E_T\rVert^2 \quad \text{in Probability.}    
\end{align*}
Moreover \begin{align*}
    \operatorname{\limsup}_{\eps\rightarrow 0}\sqrt{\eps}\left\lvert\int_0^T \langle u^{\eps},dW^{\eps}_s\rangle \right\rvert=0  \quad \text{in Probability}
\end{align*}
since by \eqref{energy estimate NS}
\begin{align*}
    \mathbb{E}\left[\left\lvert\int_0^T\langle dW^{\eps}_s,u^{\eps}_s\rangle\right\rvert\right]&\leq C\mathbb{E}\left[\operatorname{sup}_{t\in [0,T]}\lVert u^{\eps}_t\rVert^2\right]^{1/2}\leq C(N,u_0,T).
\end{align*}
Lastly 
\begin{align*}
    \operatorname{lim}_{\eps\rightarrow 0}\int_0^T\langle f^{\eps}_s, u^{\eps}_s\rangle ds=\int_0^T \langle f_s,u^E_s\rangle ds \quad \text{in Probability}.
\end{align*}
Indeed 
\begin{align*}
    \left\lvert \int_0^T\langle f^{\eps}_s, u^{\eps}_s\rangle-\langle f_s,u^E_s\rangle ds \right\rvert& \leq \left\lvert \int_0^T\langle f^{\eps}_s, u^{\eps}_s-u^E_s\rangle ds \right\rvert+\left\lvert \int_0^T\langle f^{\eps}_s-f_s, u^E_s\rangle ds \right\rvert\\ & =I_1+I_2.
\end{align*}
$I_1\rightarrow 0 $ in Probability since we assumed that $u^{\eps}\rightarrow u^E$ in $C([0,T];H)$ in Probability. $I_1\rightarrow 0 $ $\mathbb{P}-a.s.$ since in the space introduced by Skorokhod's representation theorem $f^{\eps}\rightarrow f$ weakly in $L^2(0,T;H_0)$ $\mathbb{P}-a.s.$ and $u^E\in \left(L^2(0,T;H_0)\right)^*$ $\mathbb{P}-a.s.$ Therefore we proved that \begin{align*}
    \operatorname{limsup}_{\eps\rightarrow 0}\eps \int_0^T \lVert \nabla u^{\eps}_s\rVert^2_{L^2}ds=\lVert u_0\rVert^2-\lVert u^E_T\rVert^2-2\int_0^T\langle f_s,u^E_s\rangle ds \quad \text{in Probability}
\end{align*}
which implies \autoref{strong kato hp} by \eqref{Energy Euler }.
\end{remark}

\subsection{The Case of fluids with radial symmetry}\label{sec:radial symmetry}
We now prove that an unconditioned result holds in presence of strong assumpionts on the domain and data, that is a circular symmetry, see \autoref{hypothesis symmetry}. Therefore, $D=B(0, 1)$.
%the space $$E_0 =  \Big\{ \bar u\in H(D) \ : \  \bar u(x)= u(|x|)\frac{x^\perp}{|x|} \Big\}$$ will be that of initial data, while the space of controls
%$$E := \Big\{ \bar u(t, x) \in C(0, T; H) \ : \ \bar u(t, x)= u(t,|x|)\frac{x^\perp}{|x|}\Big\}$$ endowed with the strong topology, will be our base space in which we prove the LDP. Consequently, we will need a Wiener process $\{W_t\}_{t\in [0, T]}$ taking values in $E$, which we construct by taking an orthonormal basis $\tilde\sigma_k$ of $L^2(0, 1)$ and defining 
%$$W^{RS}(t,x) = \sum_{k\in %K}\tilde\sigma_k(|x|)\frac{x^\perp}{|x|}dW^k_t$$
%for a sequence of independent, $\mathcal{F}_t$-adapted brownian motions $\{W^k_t\}_{k\in K}$. 
%Similarly to what we did earlier, we ask that 
%$$\sigma_k(|x|)\frac{x^\perp}{|x|}\in D((-A)^{1/2}) \qquad \sum_{k\in K}\|\sigma_k\frac{x^\perp}{|x|}\|^2_{ D((-A)^{1/2})}<\infty$$
%so that we assure $H_0\hookrightarrow D((-A)^{1/2})$
%\todo[inline]{si riesce a scrivere questa condizione in modo un po' meno involuto? tipo scrivendo un'ipotesi su $\sigma$ o eventualmente riuscire a rappresentare $W$ come $Q^{1/2}W$ per un certo browniano cilindrico $W$ su $L^2(0,1)$?  }

The reason why this particular geometry can be treated more easily lies in the fact the we can show that the solution of the Navier Stokes equations given by \autoref{thm well-posed ns forcing} posses radial symmetry, and in turn show that the nonlinear term in the equation vanishes. We will be able to represent the solution $u^{\eps}(t,x)=v^\eps(x)\frac{x^\perp}{|x|}$ where $v^{\eps}$ is a radial function satisfying an appropriate auxiliary equation. Then we will exploit this particular representation formula in order to prove the validity of \autoref{main thm LDP NS radial symmetry}.\\ By radial functions, we mean functions $g$ such that $g(R_{\theta}x)=g(x)$ for a.e. $x\in B,$ for each $\theta\in [0,2\pi]$, $R_{\theta}:B\rightarrow B$ being the counterclockwise rotation of the disk about its center by the angle $\theta$. Any function $u(x)$ that can be written as $\bar u(|x|)\frac{x^\perp}{|x|}$ will be called circularly symmetric and the radial function $\bar u$ will be called its radial part. 

We want to consider the following equation for the scalar function $v^\eps$:
\begin{equation}\label{auxnse}
    dv^\eps = \Big[ \eps(\Delta v^\eps - \frac{v^\eps}{|x|^2}) + \bar f_t\Big]dt + \sqrt{\eps}\sum_{k\in K}\sigma_k(|x|)dW^k_t  
\end{equation}
where the forcing $\bar f_t$ and the initial datum $\chi^\eps$ are radial functions in $L^2(B)$. In order to study problem \eqref{auxnse} we need to introduce some space of functions and operators, we refer to \cite{lopes2008vanishing} and the references therein for the proof of this statement and some discussions on this topic.\\
Let $\mathcal{H}^1:=H^1_0(D)\cap L^2(D, \frac{dx}{|x|^2})$ endowed with the following scalar product 
 $$\langle u, v \rangle_{\mathcal{H}^1}= \langle \nabla u, \nabla v \rangle + \langle \frac{u}{|x|}, \frac{v}{|x|} \rangle $$
 Define the operator $- \tilde A: D(\tilde A)\rightarrow L^2(D)$ as $-Au = g$ whenever there exist $g\in L^2(D)$ such that 
 $$\langle u, v \rangle_{\mathcal{H}^1} = \langle g, v\rangle.$$
 Then the following statement holds.
\begin{lemma} 
  The operator $\Tilde A$ generates a self-adjoint analitic semigroup of negative type $e^{\Tilde A t}$ over $L^2(D)$, $D((-\Tilde A)^{1/2})=\mathcal{H}^1$. Moreover, if $x_0$ has radial symmetry then the same holds also for $e^{\Tilde A t} x_0 \ \forall t\geq 0$.
\end{lemma}
Therefore, according to \cite{da2014stochastic}, problem \eqref{auxnse} can be interpreted in mild form as
\begin{equation}\label{mild formula aux nse}
    v_t^\eps := e^{\eps \Tilde At}\chi^\eps + \int_0^t e^{\eps \Tilde A(t-s)}\bar f_sds + \sqrt{\eps}\sum_{k\in K}\int_0^t e^{\eps \Tilde A(t-s)}\sigma_k(|x|)dW_s^k
\end{equation}
We introduce the notion of weak solution of problem \eqref{auxnse}.
\begin{definition}\label{weak sol aux def}
We say that $v^\eps$ is a weak solution of \eqref{auxnse} if
$$
v^\eps \in C_{\mathcal{F}}\left([0, T] ; L^2(D)\right) \cap L^2(\Omega,\mathcal{F},\mathbb{P};L^{2}(0,T;\mathcal{H}^1)) 
$$
and for every $\phi \in D(\Tilde A)$, we have
$$
\left\langle v^\eps(t), \phi\right\rangle=\left\langle \chi ^\eps, \phi\right\rangle+\eps \int_0^t\left\langle v^\eps_s, \tilde A \phi\right\rangle d s+ \int_0^t\langle\bar f_s, \phi \rangle ds+\sqrt{\eps} \sum_{k\in K}\left\langle s_k, \phi\right\rangle W_t^k
$$
for every $t \in[0, T], \mathbb{P}-$ a.s., where $s_k=\sigma_k(|x|)$.
\end{definition} 
Thanks to \cite[Theorem 5.4]{da2014stochastic} the mild formula \eqref{mild formula aux nse} gives the unique weak solution of \eqref{auxnse}. Indeed the following hold
\begin{lemma}
    $v^\eps \in C([0, T]; L^2(D)) \cap L^2([0, T]; \mathcal{H}^1)$ and has radial symmetry. Moreover, it is a weak solution in the sense of \autoref{weak sol aux def} of equation \eqref{auxnse}.
\end{lemma}
We introduced the problem \eqref{auxnse}, because of the representation formula for the unique solution of \eqref{Symmetry with forcing} guaranteed by the following proposition.
\begin{proposition}\label{radial nse}
The unique weak solution $u^\eps$ of the Navier-Stokes system \eqref{Symmetry with forcing} for $f=\bar f(|x|)\frac{x^\perp}{|x|} \in H_0^{RS}$, $u_0= \bar u_0(|x|)\frac{x^\perp}{|x|} \in \mathcal{E}_0^{RS}$ and with noise $W^{RS}$ is given by $\bar u^\eps(|x|)\frac{x^\perp}{|x|}$ where $\bar u^\eps$ solves equation \eqref{auxnse} with forcing $\bar f$ and initial datum $\bar u_0$
\end{proposition}
\begin{proof}
We begin by proving that for every $\phi \in C^\infty_c(D \setminus \{0\};\mathbb{R}^2)$, $b\left(\bar u^\eps_s(x)\frac{x^\perp}{|x|}, \phi, \bar u^\eps_s(x)\frac{x^\perp}{|x|}\right)=0$.
\begin{align*}
    \int_D |\bar u^\eps_s(|x|)|^2\frac{x^\perp}{|x|^2}\cdot (x^\perp \cdot \grad\phi) &=  \int_D |\bar u^\eps_s(|x|)|^2\frac{x^\perp}{|x|^2}\cdot (\nabla (x^\perp\cdot \phi) + \phi^\perp) \\
    &= \int_D \frac{|\bar u^\eps_s(|x|)|^2}{|x|^2} (x^\perp \cdot (\nabla (x^\perp\cdot \phi)) + x\cdot \phi) = I_1 + I_2
\end{align*}
Now we have 
\begin{align*}
    I_1=  \int_D \frac{|\bar u^\eps_s(|x|)|^2}{|x|^2} \div (x^\perp (x^\perp\cdot \phi)) = -\int_D \nabla \Big[\frac{|\bar u^\eps_s(|x|)|^2}{|x|^2}\Big]\cdot x^\perp (x^\perp\cdot \phi) = 0
\end{align*}
since the gradient of a radial function is always parallel to $x$.
While, if we define $V(\rho)= \int_0^\rho  \frac{|\bar u^\eps_s(r)|^2}{r}dr $, we have
\begin{align*}
    I_2 = \int_D  \frac{|\bar u^\eps_s(|x|)|^2}{|x|^2} x\cdot \phi = \int_D \nabla V(|x|)\cdot \phi = -\int_D V(|x|)\div\phi = 0.
\end{align*}
Next, we observe that $\phi\cdot \frac{x^\perp}{|x|}\in D(\tilde A)$. Since we want to prove that, neglecting the non-linear term which is zero,  
\begin{align}\label{weak ns def in proof}
    &\langle u^{\eps}_t-u_0,\phi\rangle+\int_0^t \eps \langle \nabla u^{\eps}_s,\nabla \phi\rangle ds = \int_0^t \langle f_s,\phi \rangle ds +\sqrt{\eps}\sum_{k\in K}\langle \sigma_k \frac{x^\perp}{|x|},\phi\rangle.
\end{align} 
we rewrite this as
\begin{align*}
    \int_B [\bar u^\eps_t(|x|)-\bar u_0(|x|)]\phi(x) \cdot\frac{x^\perp}{|x|}dx+ &\int_0^t\int_B \eps  \nabla [\bar u^{\eps}_s(|x|)\frac{x^\perp}{|x|}]\cdot\nabla \phi(x)dxds \\ &= \int_0^t \int_B  \bar f_s(|x|)\phi(x)\cdot \frac{x^\perp}{|x|} dx ds +\sqrt{\eps}\Big(\sum_{k\in K}\int_B \sigma_k(|x|)\phi(x) \cdot\frac{x^\perp}{|x|}dx\Big) W_t^k.
\end{align*} 
which, comparing with the \autoref{weak sol aux def}, holds true if we prove that 
\begin{equation}\label{int by parts A}
    \int_0^t \int_B \nabla [\bar u^{\eps}_s(|x|)\frac{x^\perp}{|x|}]\cdot \nabla \phi(x)dxds = \int_0^t\int_B (-\Tilde A)^{1/2}\bar u^\eps_s(|x|) (-\Tilde A)^{1/2} (\phi(x)\cdot\frac{x^\perp}{|x|})dxds.\end{equation}
This can be proven with simple calculations, upon noticing that \begin{align}\label{lapl identity}
    \Delta(\phi(x)\cdot\frac{x^\perp}{|x|})= \Delta(\phi(x))\cdot \frac{x^\perp}{|x|} + \phi(x)\cdot \frac{x^\perp}{|x|^3}-2\div\left(\frac{x\cdot\phi}{\lvert x\rvert^3}x^{\perp}\right).
\end{align}
Finally, we obtain that $u^\eps$ is a weak solution of \eqref{Symmetry with forcing} by observing that the closure of $ C^\infty_c(B \setminus \{0\};\mathbb{R}^2)$ vectors field in the $H^1$ norm, is exactly $H^1_0(B;\mathbb{R}^2)$, which implies that \eqref{weak ns def in proof} holds in particular for every $\phi \in D(A)$. 
\end{proof}
In the same manner we can prove the analogous result for the Euler system, that is 
\begin{proposition}
    Given $\chi \in \E_0^{RS}$ the unique solution of the system \eqref{euler forcing} in $C([0, T]; W^{2,4}(B; \mathbb{R}^2))\cap C([0, T]; H)$ is given by $\bar u(|x|)\frac{x^\perp}{|x|}$ where the radial function $u$ is given by 
    $$\bar u(|x|):= \chi(|x|) + \int_0^t f_s(|x|)ds.$$
\end{proposition}

\subsubsection{Condition 2}\label{subsubsec:cond1 Symmetric}
In this section we prove that the second condition in the weak convergence approach is easily fulfilled unconditionally in the case of fluids with radial symmetries. Since the proof of the validity of Condition 1 in \autoref{conditions weak convergence approach} and the lower semicontinuity of the map $u_0\rightarrow I_{u_0}^{RS}(v)$ presented in \autoref{sec:Kato assumption} is valid also in this framework, this is enough to prove \autoref{main thm LDP NS radial symmetry}.
Let $\chi^\eps \rightarrow \chi$ in $\E_0^{RS}$, which we recall being endowed with the $H^3$ norm and $f^\eps \rightarrow f$ in law as $S^N(H_0^{RS})$-random variables. We will show that for each sequence $\eps_n\rightarrow 0$, $\mathcal{G}^{\eps_n,RS}\big(\chi^{\eps_n}, {\eps_n}^{1/2} W + \int_0^{\cdot} f^{\eps_n}_s ds\big)$ converges in law to $\mathcal{G}^{0,RS}(\chi,\int_0^\cdot f_s ds)$ in the topology of $\E$. This implies the validity of the second condition in \autoref{conditions weak convergence approach}. In order to simplify the notation, we will consider $\eps>0$ in the following dropping the subscript $\eps_n$, having in mind it is a countable family. Arguing as in \autoref{subsubsec:cond2 Kato}, up to passing to a different filtered probability space, by Skorohod's representation Theorem, we can assume that the convergence of the $f^\eps$ is a.s. in the weak topology of $L^2(0, T; D((-A)^{1/2}))$ (since, by construction, we have the embedding $H_0^{RS} \rightarrow D((-A)^{1/2} )$). We call $u^\eps = \G^\eps(\chi^\eps,\int_0^\cdot f_s^\eps ds +\sqrt{\eps}W)$ and $u=\G^0(\chi, \int_0^\cdot f_s ds)$. From the results in the previous section, we have that both $u^\eps$ and $u^0$ have circular symmetry, and their radial parts are given by 
$$ \bar u_t^\eps := e^{\eps \Tilde At}\chi^\eps + \int_0^t e^{\eps \Tilde A(t-s)}\bar f^{\eps}_sds + \sqrt{\eps}\sum_{k\in K}\int_0^t e^{\eps \Tilde A(t-s)}\sigma_k(|x|)dW_s^k$$
$$\bar u_t= \chi_0 + \int_0^t\bar f_s ds$$
Actually we will prove the stronger result:
\begin{align}\label{convergence radial }
  \EE\left[ \sup_{t\in[0,T]} \lVert u^\eps_t - u_t\rVert^2 \right]\rightarrow 0
\end{align}
in the probability space defined via Skorohod representation theorem.
From our representation formula, it is sufficient to show 
$$\EE\left[ \sup_{t\in[0,T]} ||\bar u^\eps - \bar u||^2\right]\rightarrow 0 $$
To do so we write 
\begin{align}\label{difference solutions}
    \bar u^\eps_t - \bar u_t = & e^{\eps \Tilde At}(\chi^\eps -\chi) + (e^{\eps \Tilde At} - I)\chi + \int_0^t e^{\eps \Tilde A(t-s)}(\bar f^\eps_s - \bar{f}_s)ds \notag \\ + &\int_0^t (e^{\eps \Tilde A(t-s)} - I)\bar f_s +\sqrt{\eps}\sum_{k\in K}\int_0^t e^{\eps \Tilde A(t-s)}\sigma_k(|x|)dW_s^k.
\end{align}
Preliminarily we observe that in the proof of  \autoref{radial nse}, we showed equality \eqref{int by parts A} for every vector field $\phi(x) \in H^1_0$. If we disregard the time integration and let $\bar u_s^\eps$ in \eqref{int by parts A} be any generic radial function in $\HH^1$ we obtain that the map $J:\HH_R^1 \rightarrow D((-A)^{1/2})$ that sends any radial function $v(|x|)$ in $\HH^1$ to the circular symmetric vector field $v(|x|)\frac{x^\perp}{|x|}$ is an isometry (where we indicated with $\HH^1_R$ the set of radial function of $\HH^1$). We then obtain 
$$\lVert f^{\eps}\rVert_{D((-A)^{1/2})}=\|\bar f_s^\eps(|x|)\frac{x^\perp}{|x|}\|_{D((-A)^{1/2})}= \| \bar f_s^\eps(|x|)\|_{\HH^1},\quad \lVert f\rVert_{D((-A)^{1/2})}=\|\bar f_s(|x|)\frac{x^\perp}{|x|}\|_{D((-A)^{1/2})}= \| \bar f_s(|x|)\|_{\HH^1} $$
which gives exactly 
\begin{align}\label{uniform bound radial forcing}
 \int_0^T\|\bar f_s\|^2_{\HH^1}+\sup_{\eps>0}\int_0^T\|\bar f^{\eps}_s\|^2_{\HH^1} \leq  C(N).   
\end{align}
Moreover $\bar f^{\eps}\rightharpoonup \bar f$ in $L^2(0,T;\mathcal{H}^1)\quad \mathbb{P}-a.s.$ Now we can treat $\overline{u}^{\eps}-\overline{u}$.
The first and second terms in \eqref{difference solutions} go to zero in $L^2$ norm thanks to the strong convergence of $\chi^\eps$ to $\chi$ and the continuity of the semigroup. The convergence is uniform in time since, for the first term $\|\sup_{t\le T}e^{\eps \Tilde At}(\chi^\eps -\chi)\| \le \|(\chi^\eps -\chi)\|$ while for the second, we choose for every $\eps$, $t_{\eps}$ for which $\|(e^{\eps \Tilde At_{\eps}} - I)\chi\|$ achieves its maximum over $[0, T]$. Then by observing that $t_{\eps}\leq T$, we get that as $\eps \rightarrow 0$, $\eps t_{\eps}\rightarrow 0$, and we conclude by the continuity of the semigroup.
The stochastic integral term can be easily controlled using the It\"o formula and Burkolder-Davis-Gundy inequality for Stochastic Convolutions, see for example \cite{Seidler2010}, obtaining
$$\EE\left[\sup_{t\le T}\|\sqrt{\eps}\sum_{k\in K}\int_0^t e^{\eps \Tilde A(t-s)}\sigma_k(|x|)dW_s^k\|^2 \right] \le 2T\sqrt{\eps}\EE\left[\sup_{t\le T}\|\bar u^\eps_s\|^2\right] \left(\sum_{k\in K}\|\sigma_k\|^2\right)$$
which converges to zero as all the quantities are bounded. The remaining terms can be treated similarly to \autoref{convergence forcing}.
In order to study the third term, call $\bar F_t^\eps=\int_0^t \bar f_s^\eps ds $ and $\bar F_t=\int_0^t \bar f_s ds$. By \cite[Proposition 26]{flandoli2023stochastic}, \begin{align}\label{convergence integral forcing radial}
    \bar F^{\eps}\rightarrow \bar F \quad in\  C([0,T];L^2)\quad \mathbb{P}-a.s.
\end{align}  The third term in \eqref{difference solutions} can can be rewritten as 
$$ \int_0^t e^{\eps \Tilde A(t-s)}(\bar f^\eps_s - \bar f_s)ds = \eps \int_0^t \Tilde A e^{\eps \Tilde A(t-s)}(\bar F^\eps_s - \bar F_s)ds + (\bar F^\eps_t - \bar F_t).$$ The second term above converges to $0$ $\mathbb{P}-a.s.$ and in $L^2(\Omega,\mathbb{P})$ thanks to \eqref{uniform bound radial forcing} and \eqref{convergence integral forcing radial}.
Concerning the other we use standard properties of analytic semigroup, see for example \cite[Chapter 2]{pazy2012semigroups}, and \eqref{uniform bound radial forcing} obtaining 
\begin{align*}
    \EE\left[\sup_{t\in [0,T]}\lVert \eps \int_0^t \Tilde A e^{\eps \Tilde A(t-s)}(\bar F^\eps_s - \bar F_s)ds \rVert^2\right]&\leq C\eps \mathbb{E}\left[\operatorname{sup}_{t\in [0,T]}\left(\int_0^t \frac{\lVert \bar{F}^{\eps}-\bar F\rVert_{C([0,T];\mathcal{H}^1)}}{(t-s)^{1/2}}ds\right)^2\right]\\ & \leq C(N,T)\eps\rightarrow 0.
\end{align*}
%By $F^{\eps}\rVert_{C([0,T];D((-A)^2))}+\lVert F\rVert_{C([0,T];L^2)}\leq C(N)\quad {\mathbb{P}}-a.s.$
%if we now take the supremum in time of the $L^2$ norm of this expression, obtaining
%\begin{equation}\label{semigroup estimate radial 1}
%   \sup_{[0, T]}\|\int_0^t e^{\eps \Tilde A(t-s)}(\bar f^\eps_s - f_s)ds\|^2_{L^2(B)} \le 2 \sqrt{\eps} \int_0^T \frac{1}{(t-s)^{1/2}} \|\bar F^\eps_s - \bar F_s\|^2_{D((-\tilde A)^{1/2})}ds + 2\sup_{[0, T]}\|\bar F^\eps_t - \bar F_t\|^2.
%\end{equation}
%In order to prove that these passages are rigorous, we show that $\bar F^\eps_s - \bar F_s \in D((-\tilde A)^{1/2}) = \HH^1 $
%Since $f^\eps, \ f \in S^N(H_0)$ and $H_0 \hookrightarrow D((-A)^{1/2})$ by assumption, we have that 
%$$ \bar f^\eps_t(|x|)\frac{x^\perp}{|x|} \in D((-A)^{1/2})$$
%Going back to \eqref{semigroup estimate radial 1} we now take the supremum in time and observe that $\sup_{[0, T]}\|\bar F^\eps_s - \bar F_s\|^2_{D((-A)^{1/2})}$ converges to zero almost surely, since the functions $F^\eps_s \text{ and } \bar F_s$ belong to $W^{1, 2}(0, T; \0)$ and $\bar f^\eps_s \rightharpoonup \bar f_s$. Finally we can pass to the convergence in mean by dominated convergence, since the random variables $f^\eps$ and $f$ are bounded by $N$ in $L^2(0, T, H_0)$ by hypothesis, and so $\sup_{[0, T]}\|\bar F^\eps_s - \bar F_s\|^2_{D((-A)^{1/2})} \le 2N$.
Finally, for the fourth term, by standard properties of analytic semigroups, see for example \cite[Chapter 2]{pazy2012semigroups}, and our assumption on $f$, see \eqref{uniform bound radial forcing},  
\begin{align*}
    \sup_{t\in [0, T]}\|\int_0^t (e^{\eps \Tilde A(t-s)} - I)\bar f_s ds\|^2 & \le  \left(\sup_{t\in [0, T]}\int_0^t\|(e^{\eps \Tilde A(t-s)} - I)\bar f_s\| ds\right)^2  \\ 
    &\le C\eps\left(\sup_{t\in [0, T]}\int_0^t (t-s)^{1/2}\|\bar f_s\|_{\mathcal{H}^1} ds\right)^2\\ & \le C(N,T)\eps \quad \mathbb{P}-a.s.
\end{align*}
Therefore we get
 $$\EE \Big[\sup_{[0, T]}\|\int_0^t (e^{\eps \Tilde A(t-s)} - I)\bar f_s ds\|^2\Big] \rightarrow 0. $$
 This prove the validity of relation \eqref{convergence radial }.
\begin{remark}
    According to \autoref{remark inverse kato}, computation above implies in particular that \autoref{strong kato hp} are satisfied in this framework.
\end{remark}
\section{Second-Grade Fluids}\label{sec:second Grade}
We observe that the proof of the validity of Condition 1 of \autoref{conditions weak convergence approach} presented in \autoref{sec:Kato assumption} holds also in this framework. Moreover the same is true also for the lower continuity of the map $I_{u_0}^{SG}(v)$ for $v\in \mathcal{E}$ fixed. Therefore, in order to prove \autoref{main thm LDP 2GF}, it is enough to show the validity of Condition 2 of \autoref{conditions weak convergence approach}.

\subsection{Condition 2}\label{subsubsec:cond1 Second Grade}
We argue similarly to the proof of the validity Condition 2 in \autoref{sec:Kato assumption}. Fix $N>0$, let $\tilde{f}^{\eps},\Tilde{f}\in \mathcal{P}_2^N,\ u_0^{\eps},\ u_0\in \mathcal{E}_0^{SG}$ such that $\Tilde{f}^{\eps}\rightarrow_{\mathcal{L}}\Tilde{f}$ weakly in $L^2(0,T;H_0),\ u_0^{\eps}\rightarrow u_0$ in $\mathcal{E}_0^{SG}.$ We will show that for each sequence $\eps_n\rightarrow 0,\nu_n\rightarrow 0$ s.t. $\nu_n=O(\eps_n)$, $\mathcal{G}^{\eps_n,SG}\big(u_0^{\eps_n}, {\eps_n}^{1/2} W + \int_0^{\cdot} f^{\eps_n}_s ds\big)$ converges in law to $\mathcal{G}^{0,SG}(u_0,\int_0^\cdot u_s ds)$ in the topology of $\E$. This implies the validity of the second condition in \autoref{conditions weak convergence approach}. In order to simplify the notation, we will consider $\eps>0,\ \nu>0$ in the following dropping the subscript $\eps_n$, $\nu_n$, having in mind they are countable families. Since $S^N$ is a Polish space, by Skorokhod's representation theorem we can introduce a further filtered probability space $(\Tilde{\Omega},\Tilde{\mathcal{F}},\Tilde{\mathcal{F}}_t,\Tilde{\mathbb{P}})$ and random variables $f^{\eps}, W^{\eps},f$ such that $(f^{\eps},W^{\eps})$ has the same joint law of $(\Tilde{f}^{\eps},W)$, $f$ has the same law of $\Tilde{f}$ and $f^{\eps}\rightarrow_{\Tilde{\mathbb{P}}-a.s.} f$ in $L^2(0,T;H_0)$, see for example \cite{flandoli2023stochastic} for details. Thanks to \autoref{thm well-posed sf forcing} for each $\epsilon$ we can define $u^{\epsilon}$ as the unique solution of \eqref{second grade with forcing} with forcing term $f^{\eps}$, initial condition $u_0^{\eps}$ and Brownian forcing term $W^{\eps}$. Moreover, by \autoref{thm well posed euler forcing} we can define ${u}^E$ as the unique regular solution of \eqref{euler forcing} with forcing term $f$ and initial condition $u_0.$ We will show that $u_0^{\eps}$ converges to ${u}^E$ in probability in $C([0,T];H)$. This implies the validity of Condition 2.\\
Before starting with the computation we recall some facts. In the following, with some abuse of notation, we will simply use $\mathbb{P},\ \mathbb{E}$ instead of $\Tilde{\mathbb{P}},\ \Tilde{\mathbb{E}}.$ Fix $\theta>0$ arbitrarily small and define $F^{\eps}(t)=\int_0^t f^{\eps}_s ds,\ F(t)=\int_0^t f_s ds$. By hypothesis \autoref{hypothesis noise}, $H_0\hookrightarrow D((-A)^2)$. This implies, see for example \cite[Proposition 26]{flandoli2023stochastic} that 
\begin{align}\label{convergence integral forcing SG}
    F^{\eps}\rightarrow_{{\mathbb{P}}-a.s.} F \quad\textit{in }C([0,T];D((-A)^{\gamma-\theta})).
\end{align} 
Obviously
\begin{align}\label{uniform bound integral forcing SG}
\operatorname{sup}_{\eps>0}\lVert F^{\eps}\rVert_{C([0,T];D((-A)^\gamma))}+\lVert F\rVert_{C([0,T];D((-A)^\gamma))}\leq C(N)\quad {\mathbb{P}}-a.s.  
\end{align}
Starting from \eqref{ITO SG 1} and \eqref{ITO SG 2}, under \autoref{Second Grade Fluids assumption}, Burkholder-Davis-Gundy inequality, Gr\"onwall's lemma and the convergence of $u_0^{\eps}$ to $u_0$ imply
\begin{align}
    &\mathbb{E}\left[\operatorname{sup}_{t\in  [0,T]}\lVert u^{\eps}_t\rVert^2\right]+\eps\mathbb{E}\left[\operatorname{sup}_{t\in  [0,T]}\lVert \nabla u^{\eps}_t\rVert^2\right]+2\nu \mathbb{E}\left[\int_0^T \lVert \nabla u^{\eps}_s\rVert^2 ds\right]\leq C(N,u_0) \label{energy estimate second grade 1 },\\ 
    &\eps^3\mathbb{E}\left[\operatorname{sup}_{t\in  [0,T]}\lVert u^{\eps}_t\rVert_{H^3}^2\right]\leq C(N,u_0)  \label{energy estimate second grade 2},
\end{align}
see \cite[Section 6]{luongo2022inviscid} for the details.
In order to show the convergence of $u^{\eps}$ to $u^E$ we will introduce
 \begin{align*}
     z^{\eps}=\int_0^t e^{\nu (I-\eps A)^{-1}A (t-s)}(I-\eps A)^{-1}f^{\eps}_sds
 \end{align*} which is the mild solution of 
\begin{align*}
    d(I-\eps A)z^{\eps}=\nu A z^{\eps}+f^{\eps},
\end{align*} 
$ v^{\eps}=u^{\eps}-z^{\eps},\ v^E=u^E-F$ and show separately the convergence of $z^{\eps}$ to $F$ and of $v^\eps$ to $v^E.$ We recall that the the operators $A$ and $(I-\eps A)^{-1}$ commute on $D((-A)^{\alpha})$ for each $\alpha\in \mathbb{R}$. Moreover \begin{align*}
    (I-\eps A)^{-1}A: D((-A)^{\alpha})\rightarrow D((-A)^{\alpha})
\end{align*} is a linear, bounded operator for each $\alpha\in \mathbb{R}$ with operatorial norm equal to $\frac{1}{\eps}$.
We start with the convergence of $z^{\eps}$ to $F$.
\begin{lemma}\label{convergence forcing second grade}
    For each $\theta>0$, $z^{\eps}(t)\rightarrow F$ in $C([0,T];D((-A)^{\gamma-2\theta}))\quad \mathbb{P}-a.s.$ and in $L^2(\Omega,\mathbb{P})$.
\end{lemma}
\begin{proof}
$z^{\eps}$ can be rewritten as 
    \begin{align*}
        z^{\eps}_t=\nu\int_0^t (I-\eps A)^{-2}A e^{\eps (I-\eps A)^{-1}A (t-s)}F^{\eps}_s ds+(I-\eps A)^{-1}F^{\eps}_t=I_1+I_2
    \end{align*}
    Let us show that $I_2\rightarrow F$ and $I_1 \rightarrow 0$ properly. We start from $I_2:$
    \begin{align*}
        I_2=((I-\eps A)^{-1}F^{\eps}_t-F^{\eps}_t)+F^{\eps}_t=I_{2,1}+I_{2,2}
    \end{align*}
    Now $I_{2,2}\rightarrow F_t$ in $C([0,T];D((-A)^{\gamma-\theta}))$ $\mathbb{P}-a.s.$ and  $L^2(\Omega,\mathbb{P})$ by \eqref{convergence integral forcing SG} and \eqref{uniform bound integral forcing SG}. For what concerns $I_{2,1}$ we have
    \begin{align*}
        \operatorname{sup}_{t\in [0,T]}\lVert (-A)^{\gamma-2\theta}I_{2,1}\rVert&= \operatorname{sup}_{t\in [0,T]}\lVert  (-A)^{\gamma-2\theta}(F^{\eps}_t-(I-\eps A)^{-1}F^{\eps}_t)\rVert\\ & =\operatorname{sup}_{t\in [0,T]}\lVert (-A)^{\gamma-2\theta}\left(\frac{I}{\eps}-A\right)^{-1}A F^{\eps}_t\rVert \\ & = \operatorname{sup}_{t\in [0,T]}\lVert (-A)^{1-\theta}\left(\frac{I}{\eps}-A\right)^{-1}\rVert \lVert (-A)^{\gamma-\theta}F^{\eps}_t\rVert \\ & \leq \eps^{\theta}\frac{\left(\frac{1-\theta}{\theta}\right)^{1-\theta}}{1+\frac{1-\theta}{\theta}}C(N,T)\rightarrow 0 \ \mathbb{P}-a.s.
    \end{align*}
    Since previous bound is uniform in $\omega\in \Omega$, previous inequalities imply also the convergence in $L^2(\Omega,\mathbb{P})$.
   For what concerns $I_1$ we have \begin{align*}
        \operatorname{sup}_{t\in [0,T]} \lVert (-A)^{\gamma-\theta}I_1\rVert&\leq \nu \operatorname{sup}_{t\in [0,T]}\int_0^t \lVert (-A)^{1-\theta}(I-\eps A)^{-2}e^{\eps(I-\eps A)^{-1}A(t-s)}(-A)^{\gamma}F^{\eps}_s\rVert ds \\ & \leq \eps^{\theta}\frac{\left(\frac{1-\theta}{\theta+1}\right)^{\frac{1-\theta}{2}}}{1+\frac{1-\theta}{\theta+1}}C(N,T) \rightarrow 0  \ \mathbb{P}-a.s.
    \end{align*}
       Since previous bound is uniform in $\omega\in \Omega$, previous inequalities imply also the convergence in $L^2(\Omega,\mathbb{P})$. Combining the convergence of $I_1$, $I_{2,1}$ and $I_{2,2}$ the thesis follows.
\end{proof}
\begin{remark}\label{remark uni bound sg aux}
    Since $H_0\hookrightarrow D((-A)^{3/2})\hookrightarrow W$, $z^{\eps}$ satisfies relations \eqref{energy estimate second grade 1 }, \eqref{energy estimate second grade 2}. We show a stronger relation. Indeed it holds:
    \begin{align}\label{uniform energy a.s. zeps}
        \operatorname{sup}_{t\in [0,T]}\lVert (-A)^{3/2}z^{\eps}_t\rVert&\leq \operatorname{sup}_{t\in [0,T]} \int_0^t \lVert (-A)^{3/2}e^{\nu (I-\eps A)^{-1}A (t-s)}(I-\eps A)^{-1}f^{\eps}_s\rVert ds \notag \\ & \leq \operatorname{sup}_{t\in [0,T]} \int_0^t \lVert (-A)^{3/2}f^{\eps}_s\rVert ds \notag\\ & \leq C(N,T) \quad \mathbb{P}-a.s.
    \end{align}
    Since previous bound is uniform in $\omega\in \Omega$, we have also 
    \begin{align}\label{uniform energy L2}
        \mathbb{E}\left[\operatorname{sup}_{t\in [0,T]}\lVert (-A)^{3/2}z^{\eps}_t\rVert^2\right]\leq C(N,T).
    \end{align} 
\end{remark}
In order to prove the convergence of $v^{\eps}$ to $v^E$ we observe that they solve in a sense analogous to \autoref{weak solution second grade forcing}, \autoref{weak euler}
 \begin{align}\label{auxiliary eps second grade}
   d(v^{\eps}-\eps \Delta v^{\eps})=(\nu \Delta v^{\eps}-\operatorname{curl}(v^{\eps}+z^{\eps}-\eps \Delta (v^{\eps}+z^{\eps}))\times (v^{\eps}+z^{\eps})+\nabla q^{\eps})dt+\sqrt{\eps}dW_t
\end{align}
and \eqref{auxiliary 0}.
By triangle inequality and the uniform bound guaranteed by \autoref{remark uni bound sg aux}, 
 estimates analogous to \eqref{energy estimate second grade 1 }, \eqref{energy estimate second grade 2}, \eqref{Energy Euler } hold for $v^{\eps}$ and $v^E$, too. Moreover $v^E$ satisfies \eqref{energy derivative time vE}. Again, we introduce the corrector of the boundary layer $v$ of width $\delta=\delta(\eps)$, i.e. a divergence free vector field with support in a strip of the boundary of width $\delta$ such that $v^E-v\in V$ and $\mathbb{P}-a.s.$ uniformly in $t\in [0,T],\ \omega\in \Omega$ \begin{align}\label{property bl corrector 2}
   & \lVert v_t\rVert_{L^{\infty}(D)}\leq C(N,u_0), \ \lVert v_t\rVert\leq C(N,u_0)\delta^{\frac{1}{2}}, \ \lVert\partial_t v_t\rVert\leq C(N,u_0)\delta^{\frac{1}{2}},\notag\\ & \lVert\nabla v_t\rVert_{L^{\infty}(D)}\leq C(N,u_0)\delta^{-1}, 
\  \lVert\nabla v_t\rVert_{L^2(D)}\leq C(N,u_0)\delta^{-1/2},  \   \lVert\partial_t \nabla v_t\rVert\leq C(N,u_0)\delta^{-\frac{1}{2}}.
\end{align}
We choose $\delta$ such that \begin{align}\label{hp delta SG}
    \operatorname{lim}_{\eps\rightarrow 0}\delta=0,\quad  \operatorname{lim}_{\eps\rightarrow 0}\frac{\eps}{\delta}=0.
\end{align}
Now we are ready to show the convergence of $v^{\eps}$ to $v^E$.
\begin{lemma}\label{convergence v^eps v^E SG} 
    $v^{\eps}\rightarrow v^{E}$ in $C([0,T];H)$ in probability.
\end{lemma}
\begin{proof}
    Let $w^{\eps}=v^{\eps}-v^E$.  Arguing as in \cite[Theorem 9]{luongo2022inviscid} one can show that the following relations hold true:
    \begin{align}\label{main proof second grade 1}
    d\lVert w^{\eps}\rVert^2&=\eps Tr(Q)dt+\eps^3 Tr(A^2(I-\eps A)^{-2}Q)dt+\eps^2Tr(A(I-\eps A)^{-1}Q)dt \notag\\ & +2\nu \langle w^{\eps}, Av^{\eps}\rangle dt+2\sqrt{\eps}\langle w^{\eps}, dW^{\eps}_t\rangle+b(v^E+F,v^E+F, w^{\eps})dt\notag\\ & 
    -2b(v^{\eps}+z^{\eps}, (I-\eps \Delta)(v^{\eps}+z^{\eps}),w^{\eps})dt-2b(w^{\eps}, (I-\eps \Delta)(v^{\eps}+z^{\eps}),v^{\eps}+z^{\eps})dt\notag\\ 
    &+2\eps \langle w^{\eps}, d Av^{\eps}\rangle 
    \end{align}
where \begin{align*}
    \langle w^{\eps},d Av^{\eps}\rangle & =-\frac{d\lVert (-A)^{1/2}v^{\eps}\rVert^2}{2}+\frac{d\langle \langle (-A)^{1/2}v^{\eps}, (-A)^{1/2}v^{\eps}\rangle\rangle}{2}\\ &+d\langle (-A)^{1/2}(v^{E}-v),(-A)^{1/2}v^{\eps}\rangle-\langle\partial_t (-A)^{1/2}(v^E-v), (-A)^{1/2}v^{\eps}  \rangle\\ & -d\langle v,A v^{\eps}\rangle+\langle\partial_t v, A v^{\eps}  \rangle.
\end{align*}
First we observe that \begin{align}\label{main proof second grade 2}
    \eps Tr(Q)+\eps^3 Tr(A^2(I-\eps A)^{-2}Q)+\eps^2Tr(A(I-\eps A)^{-1}Q)=o(1).
\end{align}
Secondly, we can rewrite the trilinear forms as
\begin{align}\label{main proof second grade 3}
    &b(v^E+F,v^E+F, w^{\eps})
    -b(v^{\eps}+z^{\eps}, v^{\eps}+z^{\eps},w^{\eps})\notag +\eps b(v^{\eps}+z^{\eps},\Delta(v^{\eps}+z^{\eps}),w^{\eps})\notag\\ & -\eps b(w^{\eps}, \Delta(v^{\eps}+z^{\eps}),v^{\eps}+z^{\eps}) \pm b(v^{\eps},v^{E},w^{\eps})=\notag\\ & b(F,v^E+F, w^{\eps})+b(v^E, F, w^{\eps})-b(z^{\eps},v^{\eps}+z^{\eps},w^{\eps})-b(v^{\eps},z^{\eps},w^{\eps})-b(w^{\eps},v^{E}, w^{\eps})\notag \\ & +\eps b(u^{\eps},\Delta u^{\eps},w^{\eps}) -\eps b(w^{\eps}, \Delta u^{\eps},u^{\eps}).
\end{align}
Inegrating in time between $0$ and $t$ equation \eqref{main proof second grade 1} and exploiting \eqref{main proof second grade 2}, \eqref{main proof second grade 3}, 
 we get
\begin{align}\label{main proof second grade 4}
    \lVert w^{\eps}_t\rVert^2+\eps \lVert \nabla v^{\eps}_t\rVert^2 &=o(1)+\lVert u_0^{\eps}-u_0\rVert^2+\eps \lVert \nabla u_0^{\eps}\rVert^2+2\eps \langle \nabla (v^{E}_t-v_t),\nabla v^{\eps}_t\rangle  -2\eps \langle\nabla (u_0-v_0),\nabla u_0^{\eps}\rangle\notag\\ &-2\eps \int_0^t \langle \partial_s \nabla (v^E_s-v_s),\nabla v^{\eps}_s\rangle ds-2\eps \langle v_t,\Delta v^{\eps}_t\rangle +2\eps\langle v_0,\Delta u_0^{\eps}\rangle+2\eps \int_0^t \langle\partial_s v_s, \Delta v^{\eps}_s\rangle ds\notag\\ & +2\nu \int_0^t \langle w^{\eps}_s, Av^{\eps}_s\rangle ds +2\sqrt{\eps}\int_0^t \langle w^{\eps}_s, dW^{\eps}_s\rangle+2 \int_0^t b(F_s,u^E_s, w^{\eps}_s)-b(z^{\eps}_s,u^{\eps}_s,w^{\eps}_s)ds \notag \\ & +2\int_0^t b(v^E_s, F_s, w^{\eps}_s)-b(v^{\eps}_s,z^{\eps}_s,w^{\eps}_s)-b(w^{\eps}_s,v^E_s, w^{\eps}_s)ds\notag \\ & 
    +2\eps\int_0^t  b(u^{\eps}_s,\Delta u^{\eps}_s,w^{\eps}_s) - b(w^{\eps}_s, \Delta u^{\eps}_s,u^{\eps}_s) ds.
\end{align}
In order to reach our final expression for the evolution of $\lVert w^{\eps}_t\rVert^2$
we rewrite better the terms related to the forcing terms $f^{\eps},\ f$ in equation \eqref{main proof second grade 3}
\begin{align}\label{main proof second grade 5}
&b(F, u^{E},w^{\eps})+b(v^E,F,w^{\eps})-b(v^{\eps},z^{\eps},w^{\eps})-b(z^{\eps},u^{\eps},w^{\eps})\notag  \\ &= b(F,v^E,v^{\eps})+b(F,F,w^{\eps})+b(v^E,F,w^{\eps})-b(z^{\eps},z^{\eps},w^{\eps})\notag\\ & +b(z^{\eps},v^{\eps},v^E)-b(v^{\eps},z^{\eps},w^{\eps})\pm b(F,z^{\eps}, w^{\eps})\pm  b(v^E, z^{\eps},w^{\eps})\notag \\ & =b(F-z^{\eps},v^E,v^{\eps})+b(F,F-z^{\eps},w^{\eps})\notag\\ &+b(v^E, F-z^{\eps},w^{\eps})+b(F-z^{\eps},z^{\eps},w^{\eps})+b(w^{\eps},z^{\eps},w^{\eps}).
\end{align}
Since $u_0^\eps\rightarrow u_0$ in $\mathcal{E}_0^{SG}$, our choice of $\delta$, see \eqref{hp delta SG}, and the properties of the boundary layer corrector \eqref{property bl corrector 2} we have easily, see \cite[Equation (99)]{luongo2022inviscid} for details, \begin{align*}
    \eps \lVert \nabla u_0^{\eps}\rVert^2-2\eps \langle \nabla(u_0-v_0),\nabla u_0^{\eps}\rangle+\eps \langle v_0,\Delta u_0^{\eps}\rangle=o(1).
\end{align*}
Inserting \eqref{main proof second grade 5} in \eqref{main proof second grade 4} we get 
\begin{align}\label{main proof second grade 6}
    \lVert w^{\eps}_t\rVert^2+\eps \lVert \nabla v^{\eps}_t\rVert^2 &=o(1)+2\eps \langle \nabla (v^{E}_t-v_t),\nabla v^{\eps}_t\rangle  -2\eps \int_0^t \langle \partial_s \nabla (v^E_s-v_s),\nabla v^{\eps}_s\rangle ds\notag \\ &-2\eps \langle v_t,\Delta v^{\eps}_t\rangle +2\eps \int_0^t \langle\partial_s v_s, \Delta v^{\eps}_s\rangle ds +2\nu \int_0^t \langle w^{\eps}_s, Av^{\eps}_s\rangle ds +2\sqrt{\eps}\int_0^t \langle w^{\eps}_s, dW^{\eps}_s\rangle\notag\\ & +2 \int_0^t b(F_s+v^E_s,F_s-z^{\eps}_s, w^{\eps}_s)-b(F_s-z^{\eps}_s,v^E_s,v^{\eps}_s)ds+ b(F_s-z^{\eps}_s, z^{\eps}_s, w^{\eps}_s) ds\notag\\ & -2\int_0^t b(w^{\eps}_s,z^{\eps}_s,w^{\eps}_s)-b(w^{\eps}_s,v^E_s, w^{\eps}_s)ds\notag \\ & 
    +2\eps\int_0^t  b(u^{\eps}_s,\Delta u^{\eps}_s,w^{\eps}_s) - b(w^{\eps}_s, \Delta u^{\eps}_s,u^{\eps}) ds\notag \\ & =I_1(t)+I_2(t)+I_3(t)+I_4(t)+I_5(t)+I_6(t)+M(t),
\end{align}
where
\begin{align*}
    I_1(t)&=2\eps \langle \nabla (v^{E}_t-v_t),\nabla v^{\eps}_t\rangle-2\eps \langle v_t,\Delta v^{\eps}_t\rangle,\\
    I_2(t)&=-2\eps \int_0^t \langle \partial_s \nabla (v^E_s-v_s),\nabla v^{\eps}_s\rangle ds+2\eps \int_0^t \langle\partial_s v_s, \Delta v^{\eps}_s\rangle ds,\\
    I_3(t)&=2\nu \int_0^t \langle w^{\eps}_s, Av^{\eps}_s\rangle ds,\\ 
    I_4(t)&=-2\int_0^t b(w^{\eps}_s,z^{\eps}_s,w^{\eps}_s)-b(w^{\eps}_s,v^E_s, w^{\eps}_s)ds,\\ I_5(t)&=2 \int_0^t b(F_s+v^E_s,F_s-z^{\eps}_s, w^{\eps}_s)-b(F_s-z^{\eps}_s,v^E_s,v^{\eps}_s)ds+ b(F_s-z^{\eps}_s, z^{\eps}_s, w^{\eps}_s) ds,\\ I_6(t)&=2\eps\int_0^t  b(u^{\eps}_s,\Delta u^{\eps}_s,w^{\eps}_s) - b(w^{\eps}_s, \Delta u^{\eps}_s,u^{\eps}) ds,\\ M(t)&=2\sqrt{\eps}\int_0^t \langle w^{\eps}_s, dW^{\eps}_s\rangle.
\end{align*}
Equation \eqref{main proof second grade 6} is the final expression that we will use in order to estimate the various terms and apply Gr\"onwall's lemma. The analysis of $I_1(t)$ follows by the properties of the boundary layer corrector \eqref{property bl corrector 2}, our choice of $\delta$ \eqref{hp delta SG} and the interpolation inequality \eqref{interpolation estimate}. Therefore it holds:
\begin{align}\label{main proof second grade 7}
 I_1(t) & \leq \eps C(N,u_0)(1+\delta^{-1/2})\lVert \nabla v^{\eps}_t\rVert+\eps\delta^{1/2}C(N,u_0)\lVert v^{\eps}_t\rVert_{H^2}\notag \\ & \leq   \frac{\eps}{10}\lVert \nabla v^{\eps}_t\rVert^2+\eps C(N,u_0)(1+\delta^{-1})+\eps\delta^{1/2}C(N,u_0)\lVert \nabla v^{\eps}\rVert^{1/2}\lVert v^{\eps}\rVert_{H^3}^{1/2}\notag \\ & \leq o(1)+\frac{\eps}{5}\lVert \nabla v^{\eps}_t\rVert^2+\eps^3\delta \lVert v^{\eps}_t\rVert_{H^3}^2+\delta^{1/2}C(N,u_0)\notag\\ & =o(1)+\frac{\eps}{5}\lVert \nabla v^{\eps}_t\rVert^2+\eps^3\delta \lVert v^{\eps}_t\rVert_{H^3}^2.
\end{align}

The analysis of $I_2(t)$ is analogous to the \eqref{main proof second grade 7} and leads us to
\begin{align}\label{main proof second grade 8}
     I_2(t)  &\leq o(1)+\eps \int_0^t \lVert \nabla v^{\eps}_s\rVert^2 ds+\eps^3\delta \int_0^T \lVert v^{\eps}_s\rVert_{H^3}^2 ds.
\end{align}
In order to treat $I_3(t)$, we split $w^{\eps}$ in $v^{\eps},\ v^E-v$ and $v$. Then the first two terms are integrated by parts. Exploiting the properties of the boundary layer corrector \eqref{property bl corrector 2}, our choice of $\delta$ \eqref{hp delta SG}, $\nu$ \eqref{Second Grade Fluids assumption} and the interpolation inequality \eqref{interpolation estimate} it holds

\begin{align}\label{main proof second grade 9}
     I_3(t)&=-\nu \int_0^t \lVert \nabla v^{\eps}_s\rVert^2 ds-\nu \int_0^t \langle v^E_s-v_s,Av^{\eps}_s\rangle ds -\nu \int_0^t \langle v_s,Av_s^{\eps}\rangle ds   \notag \\ &
    =-\nu \int_0^t \lVert \nabla v^{\eps}_s\rVert^2 ds+\nu \int_0^t \lVert \nabla(v^E_s-v_s)\rVert \lVert\nabla v^{\eps}_s\rVert ds +\nu \int_0^t \lVert v_s\rVert \lVert v_s^{\eps}\rVert_{H^2} ds \notag\\ & \leq 
    -\nu \int_0^t \lVert \nabla v_s^{\eps}\rVert^2 ds+\nu(1+\delta^{-1/2})C(N,u_0) \int_0^t  \lVert\nabla v_s^{\eps}\rVert ds +\nu\delta^{1/2}C(N,u_0) \int_0^t \lVert \nabla v_s^{\eps}\rVert^{1/2} \lVert v_s^{\eps}\rVert_{H^3}^{1/2} ds\notag\\ & \leq \eps \int_0^t \lVert \nabla v_s^{\eps}\rVert^2 ds +\eps^3\delta\int_0^T \lVert v_s^{\eps}\rVert_{H^3}^2 ds +\eps(1+\delta^{-1})C(N,u_0)+\delta^{1/2}C(N,u_0)\notag \\ & = o(1)+\eps \int_0^t \lVert \nabla v_s^{\eps}\rVert^2 ds +\eps^3\delta\int_0^T \lVert v_s^{\eps}\rVert_{H^3}^2 ds.
\end{align}

$I_4(t)$ can be bounded easily by H\"older's inequality, obtaining 
\begin{align}\label{main proof second grade 10}
     I_4(t) & \leq \left(\lVert z^{\eps}\rVert_{L^{\infty}(0,T;W^{1,\infty})}+\lVert v^E\rVert_{L^{\infty}(0,T;W^{1,\infty})}\right)\int_0^t \lVert w^{\eps}_s\rVert^2 ds.
\end{align}
$I_5(t)$ can be handle via H\"older's inequality, exploiting the bounds available on $F$ and $v^E$, see \eqref{uniform bound integral forcing NS} and \eqref{Energy Euler }:
\begin{align}\label{main proof second grade 11}
    I_5(t)&\leq \int_0^t \lvert b(F_s+v^E_s,F_s-z^{\eps}_s, w^{\eps}_s)\rvert+\lvert b(F_s-z^{\eps}_s,v^E_s,v^{\eps}_s)\rvert + \lvert b(F_s-z_s^{\eps}, z^{\eps}_s, w^{\eps}_s)\rvert ds\notag\\ & \leq \int_0^t \lVert w^{\eps}_s\rVert \lVert F_s-z_s^{\eps}\rVert_{H^1}\lVert F_s+v_s^E\rVert_{L^{\infty}}ds +\int_0^t \lVert v_s^{\eps}\rVert \lVert F_s-z_s^{\eps}\rVert\lVert v_s^E\rVert_{W^{1,\infty}}ds + \int_0^t \lVert w_s^{\eps}\rVert \lVert F_s-z_s^{\eps}\rVert\lVert z_s^{\eps}\rVert_{W^{1,\infty}}ds\notag\\ & \leq C(N,u_0)\lVert F-z^{\eps}\rVert_{C([0,T];H^1)}\left(1+\lVert F-z^{\eps}\rVert_{C([0,T];H^1)}\right)+\int_0^t \lVert w_s^{\eps}\rVert^2 ds +\int_0^t \lVert w_s^{\eps}\rVert \lVert F_s-z_s^{\eps}\rVert\lVert z_s^{\eps}\rVert_{W^{1,\infty}}ds.
\end{align}
Now we can move to $I_6(t)$ which is the most difficult term. Preliminarily we observe that 
\begin{align*}
 \eps b(u^{\eps},\Delta u^{\eps},w^{\eps})-\eps b(w^{\eps},\Delta u^{\eps},u^{\eps})&=\eps b(v^{\eps},\Delta u^{\eps},v^{\eps})+{\eps} b(z^{\eps},\Delta u^{\eps},w^{\eps})\\ & - \eps b(v^{\eps},\Delta u^{\eps}, v^E) -{\eps} b(v^{\eps}, \Delta u^{\eps}, v^{\eps})\\ &-{\eps} b(w^{\eps},\Delta u^{\eps},{z}^{\eps})+\eps b(v^E,\Delta u^{\eps},v^{\eps} )\\ & ={\eps} b(z^{\eps},\Delta u^{\eps},w^{\eps}) - \eps b(v^{\eps},\Delta u^{\eps}, v^E)\\ & -{\eps} b(w^{\eps},\Delta u^{\eps},{z}^{\eps})+\eps b(v^E,\Delta u^{\eps},v^{\eps} ).
\end{align*}
We start considering 
 $-\eps\int_0^t b(v^{\eps}_s,\Delta u^{\eps}_s,v^E_s) ds+\eps \int_0^t b(v^E_s,\Delta u^{\eps}_s, v^{\eps}_s) ds . $
It can treated similarly to \cite[Equations (4.18)-(4.19)]{lopes2015convergence}. $-\eps\int_0^t b(v^{\eps}_s,\Delta u^{\eps}_s,v^E_s) ds$ can be integrated by parts, then it holds:

\begin{align}\label{main proof second grade 12}
    -\eps b(v^{\eps},\Delta u^{\eps},v^E) & =\eps  \int_D  v_i^{\eps}\partial_i v_j^E \partial_{kk} u^{\eps}_j \ dx \notag\\ & = -\eps \int_D \partial_k v_i^{\eps} \partial_i v^E_j \partial_k (v^{\eps}_j+z^{\eps}_j) dx-\eps \int_D  v_i^{\eps}\partial_{i,k} v^E_j \partial_k  u^{\eps}_j dx  \notag\\ & \leq o(1)+2\eps \lVert v^E\rVert_{W^{2,4}}\lVert \nabla v^{\eps}\rVert^2.
\end{align}
In the last step we use the fact that $\eps \lVert v^E\rVert_{W^{2,4}}\lVert \nabla z^{\eps}\rVert^2=o(1)\quad\mathbb{P}-a.s.$ by \autoref{convergence forcing second grade}.
For what concerns $\eps \int_0^t b(v^E_s,\Delta u^{\eps}_s, v^{\eps}_s) ds$, we split it in three terms:
\begin{align}\label{main proof second grade 13}
    \eps b(v^E,\Delta u^{\eps},v^{\eps})&= -\eps \int_D v^E\cdot \nabla v^\eps \Delta u^{\eps}dx  \notag \\ & =-\eps \int_D (v^E-v)\cdot \nabla v^\eps \Delta v^{\eps}dx-\eps \int_D v\cdot \nabla v^\eps \Delta v^{\eps}dx-\eps \int_D v^E\cdot \nabla v^\eps \Delta z^{\eps}dx \notag\\ & =J_1+J_2+J_3.
\end{align}
$J_3$ is the easiest term and can be bounded by the right hand side of \eqref{main proof second grade 12} arguing as above.
Since  $v^E-v|_{\partial D},\ v^{\eps}|_{\partial D}=0$, we can integrate by part $J_1$ repeatedly, obtaining via H\"older's inequality the following estimate:
\begin{align}\label{main proof second grade 14}
    -\eps \int_D (v^E-v)\cdot \nabla v^\eps \Delta v^{\eps}dx &=\eps \int_D \partial_k (v_i^E-v_i)\partial_i v_j^{\eps}\partial_k v^{\eps}_j dx +\frac{\eps}{2}\int_D (v_i^E-v_i)\partial_i\lvert \partial_k v^{\eps}_j\rvert^2 dx\notag\\ & = \eps \int_D \partial_k v_i^E\partial_i v_j^{\eps}\partial_k v^{\eps}_j dx-\eps \int_D \partial_k v_i\partial_i v_j^{\eps}\partial_k v^{\eps}_j dx \notag\\ & \leq \eps \lVert v^E\rVert_{W^{2,4}}\lVert \nabla v^{\eps}\rVert^2-\eps \int_D \partial_k v_i\partial_i v_j^{\eps}\partial_k v^{\eps}_j dx\notag\\ &=\eps \lVert v^E\rVert_{W^{2,4}}\lVert \nabla v^{\eps}\rVert^2+\eps \int_D \partial_k v_i v_j^{\eps}\partial_{i,k} v^{\eps}_j dx\notag\\ & =
    \eps \lVert v^E\rVert_{W^{2,4}}\lVert \nabla v^{\eps}\rVert^2-\frac{\eps}{2} \int_D  v_i \partial_i\lvert\partial_k v_j^{\eps}\rvert^2 dx-\eps \int_D  v_i v_j^{\eps}\partial_{i,k,k} v^{\eps}_j dx \notag \\ &  = \eps \lVert v^E\rVert_{W^{2,4}}\lVert \nabla v^{\eps}\rVert^2+\eps \int_D v_i \partial_i v^{\eps}_j \partial_{k,k} v^{\eps}_j dx \notag\\ &  =\eps \lVert v^E\rVert_{W^{2,4}}\lVert \nabla v^{\eps}\rVert^2-J_2. 
\end{align}
Combining \eqref{main proof second grade 12}, \eqref{main proof second grade 13} ,\eqref{main proof second grade 14} we get 
\begin{align}\label{main proof second grade 15}
    -\eps\int_0^t b(v_s^{\eps},\Delta u_s^{\eps},v^E_s) ds+\eps \int_0^t b(v^E_s,\Delta u^{\eps}_s, v^{\eps}_s) ds & \leq o(1)+ \eps C (N,u_0)\lVert v^{E}\rVert_{L^{\infty}(0,T;W^{2,4})}\int_0^t \lVert \nabla v^{\eps}_s\rVert^2 ds.
\end{align}
We left to estimate ${\eps} \int_0^t  b(z^{\eps}_s,\Delta u^{\eps}_s,w^{\eps}_s)ds  -{\eps} \int_0^t b(w^{\eps}_s,\Delta u^{\eps}_s,{z}^{\eps}_s) ds.$
We start considering ${\eps} \int_0^t  b(z^{\eps}_s,\Delta u^{\eps}_s,w^{\eps}_s)ds$ integrating by parts repeatedly since $z^{\eps}|_{\partial D}=0$, we obtain by H\"older's inequality and the $\mathbb{P}-a.s.$ estimates on $z^{\eps}$ and $v^E$ guaranteed by \autoref{convergence integral forcing SG} and \eqref{Energy Euler } \begin{align}\label{main proof second grade 16}
   \eps b(z^{\eps},\Delta u^{\eps}, w^{\eps})& =-\eps \int_D  z^{\eps}_i \partial_i w^{\eps}_j \partial_{k,k} u^{\eps}_j dx \notag \\ & =\eps \int_D  \partial_k z^{\eps}_i \partial_i w^{\eps}_j \partial_{k} u^{\eps}_j dx+\eps \int_D  z^{\eps}_i \partial_{i,k} w^{\eps}_j \partial_{k} u^{\eps}_j dx\notag \\ &\leq \eps \lVert z^{\eps}\rVert_{W^{2,4}}(\lVert \nabla v^{\eps}\rVert+\lVert \nabla z^{\eps}\rVert)(\lVert \nabla v^{\eps}\rVert+\lVert \nabla v^E\rVert) +\eps \int_D z_i^{\eps}\partial_{i,k}(u_j^{\eps}-z_j^{\eps}-v_j^E) \partial_k u^{\eps}_j dx \notag\\ & \leq o(1)+ \eps \lVert z^{\eps}\rVert_{W^{2,4}}\lVert \nabla v^{\eps}\rVert^2+\frac{\eps}{2}\int_D z_i^{\eps}\partial_i \lvert \partial_k u^{\eps}_k\rvert^2 dx +\eps \lVert z^{\eps}\rVert_{L^4}(\lVert z^{\eps}\rVert_{W^{2,4}}+\lVert v^E\rVert_{W^{2,4}})(\lVert \nabla v^{\eps}\rVert+\lVert \nabla z^{\eps}\rVert) \notag \\ & \leq o(1)+\eps\left( \lVert z^{\eps}\rVert_{W^{2,4}}+\lVert v^{E}\rVert_{W^{2,4}}\right)\lVert \nabla v^{\eps}\rVert^2.
\end{align}
Lastly we consider $-{\eps} \int_0^t b(w^{\eps}_s,\Delta u^{\eps}_s,{z}^{\eps}_s) ds$. Here we want again integrate by parts repeatly, for this reason we add and subtract ${\eps} \int_0^t b(v_s,\Delta u^{\eps}_s,{z}^{\eps}_s) ds$ exploiting the fact that $w^{\eps}+v|_{\partial D}=0$. Therefore, thanks to the properties of the boundary layer corrector \eqref{property bl corrector 2} and computations already performed we obtain: \begin{align}\label{main proof second grade 17} 
    -\eps b(w^{\eps},\Delta u^{\eps},{z}^{\eps}) & =\eps \int_D (w_i^{\eps}+v_i)\partial_i z^{\eps}_k \partial_{j,j}u^{\eps}_k dx -\eps \int_D  v_i\partial_i z^{\eps}_k \partial_{j,j}u^{\eps}_k dx\notag \\ &= -\eps \int_D \partial_j(w_i^{\eps}+v_i)\partial_i z^{\eps}_k \partial_{j}u^{\eps}_k dx -\eps\int_D (w_i^{\eps}+v_i)\partial_{i,j} z^{\eps}_k \partial_{j}u^{\eps}_j dx + \eps \lVert v\rVert \lVert z^{\eps}\rVert_{W^{2,4}}(\lVert  v^{\eps}\rVert_{H^2}+\lVert  z^{\eps}\rVert_{H^2} ) \notag \\ & \leq o(1) +\eps C(N,u_0)\delta^{1/2}\lVert z^{\eps}\rVert_{W^{2,4}}\lVert \nabla v^{\eps}\rVert^{1/2}\lVert  v^{\eps}\rVert_{H^3}^{1/2}\notag \\ &+2\eps \lVert  z^{\eps}\rVert_{W^{2,4}} (\lVert \nabla v^{\eps}\rVert +\lVert \nabla v^{E}\rVert+\lVert \nabla v\rVert  )(\lVert \nabla v^{\eps}\rVert+\lVert \nabla z^{\eps}\rVert)\notag \\ & \leq o(1)+\eps C\lVert z^{\eps}\rVert_{W^{2,4}}\lVert \nabla v^{\eps}\rVert^2 +\eps^3 \delta\lVert v^{\eps}\rVert_{H^3}^2+\delta^{1/2}C(N,u_0)\lVert z^{\eps}\rVert_{W^{2,4}}^{3/2}+C(N,u_0)\eps\delta^{-1}\lVert z^{\eps}\rVert_{W^{2,4}}^4.
\end{align}
In conclusion, combining \eqref{main proof second grade 15},\eqref{main proof second grade 16},\eqref{main proof second grade 17}  we obtain
\begin{align}\label{main proof second grade 18}
    I_6(t) &\leq o(1) + \eps C\left( \lVert z^{\eps}\rVert_{L^{\infty}(0,T;W^{2,4})}+\lVert v^{E}\rVert_{L^{\infty}(0,T;W^{2,4})}\right)\int_0^t \lVert \nabla v^{\eps}_s\rVert^2 ds\notag + \eps^3 \delta\operatorname{sup}_{t\in [0,T]}\lVert v^{\eps}_t\rVert_{H^3}^2\notag \\ & +\delta^{1/2}C(N,u_0)\lVert z^{\eps}\rVert_{L^{\infty}(0,T;W^{2,4})}^{3/2}+C(N,u_0)\eps\delta^{-1}\lVert z^{\eps}\rVert_{L^{\infty}(0,T;W^{2,4})}^4 .
\end{align}
Combining the various estimates on the $I_i(t),\ i\in \{1,\dots, 6\}$ we get 
\begin{align}\label{main proof second grade 19}
    \lVert w^{\eps}_t\rVert^2+\frac{4}{5}\eps \lVert \nabla v^{\eps}_t\rVert^2 &\leq o(1)+C\eps(1+\lVert v^E\rVert_{L^{\infty}(0,T;W^{2,4})}+\lVert z^{\eps}\rVert_{L^{\infty}(0,T;W^{2,4})}
    )\int_0^t \lVert \nabla v^{\eps}_s\rVert^2 ds \notag\\ & +C(T)\eps^3\delta \operatorname{sup}_{t\in [0,T]}\lVert v^{\eps}_t\rVert_{H^3}^2+C\sqrt{\eps}\operatorname{sup}_{t\in [0,T]}\left\lvert\int_0^t \langle w^{\eps}_s, dW_s\rangle\right\rvert\notag\\ &+ C(1+\lVert z^{\eps}\rVert_{L^{\infty}(0,T;W^{1,\infty})}+\lVert v^E\rVert_{L^{\infty}(0,T;W^{1,\infty})})\int_0^t \lVert w^{\eps}_s\rVert^2 ds \notag \\ & +C(N,u_0)\lVert F-z^{\eps}\rVert_{C([0,T];H^1)}\left(1+\lVert F-z^{\eps}\rVert_{C([0,T];H^1)}\right) +\int_0^t \lVert w^{\eps}_s\rVert \lVert F_s-z_s^{\eps}\rVert\lVert z^{\eps}_s\rVert_{W^{1,\infty}}ds.
\end{align}
Applying Gr\"onwall's Lemma to \eqref{main proof second grade 19} we obtain
\begin{align}\label{main proof second grade 22}
    \operatorname{sup}_{t\in [0,T]}\lVert w^{\eps}_t\rVert^2+ \eps\operatorname{sup}_{t\in [0,T]}\lVert \nabla v^{\eps}_t\rVert^2 & \leq e^{C(T)(1+\lVert v^E\rVert_{L^{\infty}(0,T;W^{2,4})}+\lVert z^{\eps}\rVert_{L^{\infty}(0,T;W^{2,4})}}\times\notag \\ & \bigg (o(1)+T\eps^3\delta \operatorname{sup}_{t\in [0,T]}\lVert v^{\eps}_t\rVert_{H^3}^2 \notag\\ &+C(N,u_0)\lVert F-z^{\eps}\rVert_{C([0,T];H^1)}\left(1+\lVert F-z^{\eps}\rVert_{C([0,T];H^1)}\right) \notag\\ & +\int_0^T \lVert w^{\eps}_s\rVert \lVert F_s-z_s^{\eps}\rVert\lVert z_s^{\eps}\rVert_{W^{1,\infty}}ds +C\sqrt{\eps}\operatorname{sup}_{t\in [0,T]}\left\lvert\int_0^t \langle w^{\eps}_s, dW_s\rangle\right\rvert \bigg)
\end{align}
Under our assumptions we have $e^{CT(1+\lVert v^E\rVert_{L^{\infty}(0,T;W^{2,4})}+\lVert z^{\eps}\rVert_{L^{\infty}(0,T;W^{2,4})})}\leq C(N,u_0)\quad \mathbb{P}-a.s.$, see \eqref{Energy Euler }, \eqref{uniform energy a.s. zeps}. This means that, in order to show that \autoref{convergence v^eps v^E SG}  holds, it is enough to prove that 
\begin{align*}
    &T\eps^3\delta \operatorname{sup}_{t\in [0,T]}\lVert v^{\eps}_t\rVert_{H^3}^2 +C(N,u_0)\lVert F-z^{\eps}\rVert_{C([0,T];H^1)}\left(1+\lVert F-z^{\eps}\rVert_{C([0,T];H^1)}\right) \notag\\ & +\int_0^T \lVert w^{\eps}_s\rVert \lVert F_s-z_s^{\eps}\rVert\lVert z_s^{\eps}\rVert_{W^{1,\infty}}ds +C\sqrt{\eps}\operatorname{sup}_{t\in [0,T]}\left\lvert\int_0^t \langle w^{\eps}_s, dW_s\rangle\right\rvert\rightarrow 0\quad \text{in Probability.}
\end{align*}
Thanks to \eqref{energy estimate second grade 2}, we have $T\eps^3\delta \operatorname{sup}_{t\in [0,T]}\lVert v^{\eps}_t\rVert_{H^3}^2\rightarrow 0$ in probability.
$C(N,u_0)\lVert F-z^{\eps}\rVert_{C_tH^1_x}^2 \rightarrow 0$ in probability by \autoref{convergence forcing second grade}. Lastly, by  \autoref{convergence forcing second grade} we have also
\begin{align*}
    &\mathbb{E}\left[\sqrt{\eps}\operatorname{sup}_{t\in [0,T]}\left\lvert\int_0^t \langle w^{\eps}_s, dW_s\rangle\right\rvert+\int_0^T \lVert w_s^{\eps}\rVert \lVert F_s-z_s^{\eps}\rVert\lVert z_s^{\eps}\rVert_{W^{1,\infty}}ds\right]\\ & \leq C\mathbb{E}\left[\int_0^T \lVert w^{\eps}_s\rVert^2\right]^{1/2}\left(\sqrt{\eps}+\mathbb{E}\left[\int_0^T \lVert F_s-z_s^{\eps}\rVert^2 ds\right]^{1/2}\right)\rightarrow 0.
\end{align*}
Now the proof is complete.
\end{proof}
Combining \autoref{convergence forcing second grade} and \autoref{convergence v^eps v^E SG}  the second condition in \autoref{conditions weak convergence approach} holds. Therefore we can apply \autoref{weak convergence approach abstract thm} and complete the proof of \autoref{main thm LDP 2GF}.
\section{Some Remarks on the Kato Condition}\label{sec:Remark Kato}
We end this work with a discussion on the Kato-type condition that we assumed in order to prove one of our main results, \autoref{main thm LDP NS}. Recall that the condition \autoref{strong kato hp} was the following
\begin{hypothesis}[Strong Kato Hypothesis]
 For each $N\in \mathbb{N}$, $u_0^{\eps},\ u_0\in \mathcal{E}_0^{NS}$ and $f^{\eps},\ f\in \mathcal{P}_2^N$ such that $u_0^{\eps}\rightarrow u_0$ in $\mathcal{E}_0^{NS}$ and $f^{\eps}\rightarrow_{\mathcal{L}} f$ in $S^N$, if $(\Omega,\mathcal{F},\mathcal{F}_t,\mathbb{P})$  is a filtered probability space where all $f^{\eps}$, $f$ are defined together and $f^{\eps}\rightarrow f \ \mathbb{P}-a.s.$ in $S^N$, then,  it exists $c>0$ such that for every $\delta>0$
\begin{align*}
    \mathbb{P}\left(\eps\int_0^T\left\lVert \nabla \mathcal{G}^{NS,\eps}\left(u_0^{\eps}, \sqrt{\eps}W_\cdot+\int_0^\cdot f^{\epsilon}_s ds \right) \right\rVert_{L^2(\Gamma_{c\eps})}^2 ds>\delta\right)\rightarrow 0.
\end{align*}    
\end{hypothesis}
Loosely speaking, this condition requires a control on the behaviour in the boundary layer of the solutions of the stochastic Navier-Stokes system with respect to all kind of forcings. This is much more than what we ask to ensure the validity of the inviscid limit, namely 
\begin{align}\label{weak kato cond}
    \EE\left(\eps\int_0^T\left\lVert \nabla \mathcal{G}^{NS,\eps}\left(u_0^{\eps}, \sqrt{\eps}W_\cdot \right) \right\rVert_{L^2(\Gamma_{c\eps})}^2 \right)\rightarrow 0,
\end{align}    
(see \cite{luongo2021inviscid}). In the following, we will call the property described by equation \eqref{weak kato cond} as weak Kato condition. One can ask if this condition alone is enough to ensure the validity of the LDP. Let us notice first that what we have proved is not only a large deviation result for the Navier-Stokes system with zero forcing, $u^\eps := \mathcal{G}^{NS,\eps}\left(u_0, \sqrt{\eps}W_\cdot \right)$, but actually we got, as a byproduct, a LDP result for solutions with any forcing in $L^2(0, T; H_0)$, as pointed out in \autoref{indifference to forcing}. Indeed if we want to include in the system a forcing $g_t\in L^2(0, T; H_0)$ we can just redefine the maps $\G^0_g:=\G^0(\cdot, \int_0^\cdot(\cdot + g_s)ds)$ and $\G^\eps_g:=\G^\eps(\cdot, \int_0^\cdot(\cdot + g_s)ds)$ and the result that we have proved immediately imply a LDP for the solution with forcing, under the same \autoref{strong kato hp}.
In this sense our condition is optimal for our setting: we ask controls for every forcing and we get a LDP for every forcing term. %There are two ways for improving our result: one way could be to prove that the strong Kato hypotesis can be deduced from the  weak Kato condition. In some sense this would requires to be able to pass information between systems with different forcings. We shall notice that the forcing that we are working with all live in the reproducing kernel of $W$, therefore we might switch from one system to another just by a Girsanov transformation; however this correction explodes exponentially fast in the limit $\eps \rightarrow 0$. A posteriori, if one is able to prove that the LDP holds for some forcing, the explosion of the Girsanov correction might get compensated by the exponential decay of the law of the solutions.  

Therefore, one way of improving our result could be to prove that the strong Kato hypotesis can be deduced from the  weak Kato condition. In some sense this would requires to be able to pass information between systems with different forcings. We shall notice that the forcing that we are working with all live in the reproducing kernel of $W$, therefore we might switch from one system to another just by a Girsanov transformation; however this correction explodes exponentially fast in the limit $\eps \rightarrow 0$. A posteriori, if one is able to prove that the LDP holds for some forcing, one expects that the explosion of the Girsanov correction gets compensated by the exponential decay of the law of the solutions. A different approach would be to prove that one does not in fact need to ask the strong Kato condition in order to prove only a LDP for \autoref{NS introduction} (that is, only for the system with zero forcing). 
We formulate then the following:
\begin{conjecture}(LDP under Weak Kato Hypotesys)
The statement of \autoref{main thm LDP NS} holds true if we replace \autoref{strong kato hp} with the Kato Condition \autoref{weak kato cond}.     
\end{conjecture}

If the answer to this problem was positive, then we would expect to retrieve also the `full' LDP, that is, a family of LDP for the system with any forcing $f\in L^2(0, T; \0)$. This requires to be able to pass a LDP between systems with different forcings.
To see why this seems so natural, observe that every time one is able to write a family of solution $X_f^{\eps}$ to some S(P)DE depending by some forcing $f$ as a continuous transformation of a Brownian motion $J(\sqrt{\eps} W_.)$, then by an application of the contraction principle one immediately obtains a LDP for every other forcing. In our setting however, there seems to be no way of proving such a property of the LDP without requiring the strong Kato Condition, that is without having information about the convergence of the systems with different forcings when $\eps \rightarrow 0$. Note that, by the uniqueness of the solution to the Euler system in our setting, this convergence is also necessary condition for the LDP. \\
We then formulate also the following 
\begin{conjecture}
  The Large Deviations of our system hold independently of the choice of the forcing $f \in L^2(0, T; H_0)$, that is, if a LDP holds for at least one such forcing, then it holds for every other.    
\end{conjecture}

\begin{acknowledgements}
    We thank Professor Franco Flandoli for useful discussions and valuable insight into the subject.
\end{acknowledgements}

\bibliography{bibl}{}
\bibliographystyle{plain}

\end{document}